\documentclass[lettersize,journal]{IEEEtran}
\usepackage{cite}
\usepackage{amsmath,amssymb,amsfonts}
\usepackage{graphicx}
\usepackage{textcomp}
\usepackage{float}
\usepackage{multicol}
\usepackage{amsmath}
\usepackage{amsthm}
\usepackage{amssymb}
\usepackage{amsfonts}
\usepackage{graphicx}
\usepackage{subfigure}
\usepackage{url}
\usepackage{booktabs}
\usepackage{float}
\usepackage{color}
\usepackage{bm}
\usepackage{tabularx}
\usepackage{algorithm}
\usepackage{algpseudocode}
\usepackage{makecell}

\usepackage{tikz}
\usetikzlibrary{shapes, arrows}

\usepackage{multicol}
\usepackage{multirow}

\newcommand{\col}{\hbox{col}}
\newtheorem{assmp}{\bf Assumption}

\newtheorem{rem}{\bf Remark}

\newtheorem{lem}{\bf Lemma}

\newtheorem{thm}{\bf Theorem}

\newtheorem{prop}{\bf Proposition}

\newcommand{\EQ}{\begin{eqnarray}}
	\newcommand{\EN}{\end{eqnarray}}
\newcommand{\EQQ}{\begin{eqnarray*}}
	\newcommand{\ENN}{\end{eqnarray*}}

\ifCLASSINFOpdf
\else
\fi

\begin{document}
	\title{A New Approach to the Data-Driven Output-Based LQR Problem of Continuous-Time Linear Systems}

	\author{Liquan~Lin, Haoyan~Lin~and~Jie~Huang,~\IEEEmembership{Life Fellow,~IEEE}
		\thanks{This work was supported in part by the Research Grants Council of the Hong Kong Special Administration Region under grant No. 14202425.}
\thanks{The authors are  the Department of Mechanical and Automation Engineering, The Chinese University of Hong Kong, Hong Kong (e-mail: lqlin@mae.cuhk.edu.hk; hylin@mae.cuhk.edu.hk; jhuang@mae.cuhk.edu.hk. Corresponding author: Jie Huang.)}}
	

	\maketitle
	
	\begin{abstract}
    A promising method for constructing a data-driven output-feedback control law involves the construction of a model-free observer.
   The {\color{black}Linear Quadratic Regulator (LQR)} optimal control policy can then be obtained by both policy-iteration (PI) and value-iteration (VI) algorithms. However, this method requires some unknown parameterization matrix to be
   of full row rank and needs to solve a sequence of high dimensional linear equations for either PI or VI algorithm. In this paper, we first show that this matrix is of full row rank under the standard controllability assumption of the plant, thus removing the main hurdle for
   applying this method.
   Then we further modify the existing method by defining an ancillary system whose LQR solution will lead to a data-driven output-feedback control law. By this new method,
   the rank condition of the unknown parameterization matrix is not needed any more. Moreover,  we derive  a new sequence of linear equations  for either PI or VI algorithm whose
   number of unknown variables is significantly less than that of the existing PI or VI algorithm, thus  not only improving the computational efficiency but also relaxing the solvability conditions of the existing PI or VI algorithm. Further, since the existing  PI or VI algorithm only applies to the case where  some Riccati equation admits a unique positive definite solution, we {\color{black}prove} that both the PI algorithm   and the VI algorithm in the literature can be generalized to the case where the Riccati equation only admits a unique positive semi-definite solution.

	\end{abstract}
	

	\IEEEpeerreviewmaketitle

    \section{Introduction}

    As a  data-driven technique, reinforcement learning (RL) or adaptive dynamic programming (ADP) has proven to be effective in addressing a wide range of optimal control problems with partially or completely unknown dynamics \cite{sutton2018reinforcement}\cite{bertsekas2019reinforcement}\cite{lewis2012}.
    For instance, \cite{vrabie2009adaptive} studied the {\color{black}Linear Quadratic Regulator (LQR)} problem for partially unknown linear systems by devising a policy-iteration (PI) method to iteratively solve an algebraic Riccati equation online. \cite{jiang2012computational} extended the results of \cite{vrabie2009adaptive} to completely unknown linear systems using a PI-based approach.
    {\color{black}Recently, \cite{lopez2023} proposed a more computationally efficient PI-based algorithm over the algorithm in \cite{jiang2012computational}.}
    As the PI-based method necessitates a stabilizing feedback gain to initiate the iteration process, \cite{bian2016} introduced a value-iteration (VI) technique to address the LQR problem for unknown linear systems without needing an initial stabilizing feedback gain. Based on the PI method proposed by \cite{jiang2012computational}, \cite{gao2016adaptive} further studied the data-driven optimal output regulation problem for linear systems with unknown state equation. Then, \cite{jiang2022value} applied the VI method of \cite{bian2016} to solve the optimal output regulation problem. More recently, the results of \cite{gao2016adaptive} were further improved by \cite{lin2023}. Other relevant work can be found in, for example, \cite{chen2022homotopic,xie2022optimal,lin2024auto,chen2020off}.  However, all papers mentioned above assumed the availability of
   the state of the system.

    Since,  in more practical scenarios, only the output of the system is available, it is interesting to further engage studies on the RL/ADP methods based on the measurable output information.
    The first  attempt on solving the data-driven output-based LQR problem for continuous-time linear systems was made in \cite{zhu2014} for systems with known input matrix.   \cite{modares2016} proposed a data-driven output-based RL algorithm to solve the LQR problem for completely unknown continuous-time linear system. However, the method of \cite{modares2016} introduced  a discounted cost function whose discount factor has an upper bound determined by the system model for ensuring stability.
    On the basis of the state-based ADP algorithms in \cite{jiang2012computational} and \cite{bian2016}, \cite{rizvi2019} and  \cite{rizvi2023}
    developed both PI-based and VI-based output-based ADP algorithms that involves the construction of a model-free observer, thus
    avoiding the use of discounted cost functions.  More recently,  \cite{cui2023} proposed an output-based ADP algorithm that utilizes the historical continuous-time input-output trajectory data to reconstruct the current state without using a state observer at the cost of complicated function approximation with neural networks. Other relevant results on the output-based data-driven optimal control can be found in, for example, \cite{xie2024auto,xie2023output,chen2023output}.

   The approach of \cite{rizvi2019} and  \cite{rizvi2023} appears quite attractive as it leverages the well-known separation principle. Nevertheless, this method requires some  parameterization matrix to be
   of full row rank.  As this matrix is unknown, it is difficult to determine the rank of this matrix. Also, for a system with $n$-dimensional state, $m$-dimensional input and $p$-dimensional output, the dimension of the model-free  observer is  $n_{\zeta} = n (p+m)$. As a result, the implementation of either PI algorithm or VI algorithm will entail solving a sequence of linear algebraic equations with $\frac{n_\zeta(n_\zeta+1)}{2} + m n_\zeta$ unknown variables.
   For these reasons, in this paper,
   we first show that the unknown parameterization matrix is of full row rank under the standard controllability assumption of the given plant, thus removing the main hurdle for
   applying this method.  Then,
   we further modify the existing method by defining an ancillary linear system which is stabilizable and detectable, and show that the state-based PI and VI algorithms for solving the LQR problem
   of this ancillary linear system lead to the solution of the output-based PI and VI algorithm for solving the LQR problem of the original system.
   By this new method,
   the rank condition of the unknown parameterization matrix is not needed any more. Moreover, since the input matrix of the ancillary linear system is known,
   we obtain  a new sequence of linear equations  for either PI or VI algorithm whose
   number of unknown variables is equal to $\frac{n_\zeta(n_\zeta+1)}{2}$, thus reducing the number of unknown variables of the existing algorithms by $m n_\zeta$. As the number of unknown variables
   is reduced significantly, the solvability  condition  of these equations is also weakened significantly.
   In fact, we will apply our  modified output-based ADP PI and VI algorithms to two examples and show that these two examples  fail the  output-based ADP PI and VI algorithms in \cite{rizvi2019} and  \cite{rizvi2023}.
   {\color{black}Furthermore,  the existing state-based PI or VI algorithm only applies to the case where  some Riccati equation admits a unique positive definite solution. However,
   when it comes to the
   output-based PI or VI algorithm, one can only guarantee that
   the solution of the Riccati equation is positive semi-definite. To fill this gap, we further show that the PI algorithm as given in \cite{Kleinman}  and the VI algorithm as given in
   \cite{bian2016} can be generalized to the case where  the Riccati equation only admits a unique positive semi-definite solution.
  This result has independent interest because a Riccati equation satisfying  stabilizabity and detectability condition only admits a unique positive semi-definite solution.}

   The rest of this paper is organized as follows. In Section \ref{sec2}, after reviewing the state-based ADP algorithms based on \cite{jiang2012computational} and \cite{bian2016} as well as the output-based ADP algorithms based on \cite{rizvi2019} and \cite{rizvi2023}, we {\color{black}prove} that the unknown parameterization matrix given in \cite{rizvi2019} and \cite{rizvi2023} is of full row rank under the standard controllability assumption of the given plant.
   In Section \ref{sec3}, we present a modified approach to deal with the output-based ADP algorithms and develop two improved output-based ADP algorithms. In Section \ref{sec5}, we give two examples to illustrate the efficiency of our proposed algorithms, in which the existing output-based ADP algorithms fail.
   The paper is closed in Section \ref{sec6} with some concluding remarks. In Appendices A and B, we further study the PI method and VI method to iteratively solve the Riccati equation which only admits a unique positive semi-definite solution, respectively.

    \indent \textbf{Notation} Throughout this paper, $\mathbb{R}, \mathbb{R}_+, \mathbb{N} , \mathbb{N}_+$, $\mathbb{C}$ and $ \mathbb{C}_-$ represent the sets of real numbers, nonnegative real numbers, nonnegative integers,  positive integers, complex numbers and the open left-half complex plane, respectively. $\mathcal{S}^n$ is the set of all $n\times n$ real, symmetric matrices. $\mathcal{P}^n$ is the set of all $n\times n$ real, symmetric and positive semidefinite matrices. $||\cdot||$ represents the Euclidean norm for vectors and the induced norm for matrices. $\otimes$ denotes the Kronecker product.  For $b=[b_1, b_2, \cdots, b_n]^T\in \mathbb{R}^n$, $\text{vecv}(b)=[b_1^2,b_1b_2,\cdots,b_1b_n,b_2^2,b_2b_3,\cdots, b_{n-1}b_n,b_n^2]^T \in \mathbb{R}^{\frac{n(n+1)}{2}}$. For a symmetric matrix $P=[p_{ij}]_{n\times n}\in \mathbb{R}^{n\times n}$, $\text{vecs}(P)=[p_{11},2p_{12},\cdots,2p_{1n},p_{22},2p_{23},\cdots, 2p_{n-1,n}, p_{nn}]^T\in \mathbb{R}^{\frac{n(n+1)}{2}}$. For  $v\in \mathbb{R}^n$, $|v|_P=v^TPv$. For column vectors $a_i, i=1,\cdots,s$,  $\mbox{col} (a_1,\cdots,a_s )= [a_1^T,\cdots,a_s^T  ]^T,$ and, if  {\color{black}$A = [a_1,\cdots,a_s ]$}, then  vec$(A)=\mbox{col} (a_1,\cdots,a_s )$.
    For $A\in \mathbb{R}^{n\times n}$, $\sigma(A) $ denotes the set composed of all the eigenvalues of $A$.
    {\color{black}For matrices $A_i, i=1,\cdots,n$,
    $\mbox{blockdiag}(A_1,...,A_n)$ is the block diagonal matrix $
    \left[\begin{array}{ccc}
    	A_1&&\\
    	&\ddots&\\
    	&&A_n
    \end{array}\right]$.}
    $I_n $ denotes the identity matrix of dimension $n$. For any $0<T<\infty$, $\mathcal{D}([0,T], \mathcal{S}^n)$ denotes the space of function {\color{black}mapping} from $[0,T]$ to $\mathcal{S}^n$, that are right-continuous with left-hand limits, equipped with the Skorokhod topology. {\color{black}For a positive semidefinite matrix $Q\in \mathcal{P}^n$, $\sqrt{Q}$ denotes the unique positive semidefinite matrix such that $\sqrt{Q}^T\sqrt{Q}=Q$ and is called the square root of $Q$.}

    \section{\color{black}Preliminaries} \label{sec2}

    {\color{black}In this section, we first summarize the state-based PI approach proposed in \cite{jiang2012computational} and the state-based VI approach proposed in \cite{bian2016} for solving the model-free LQR problem.  Then, the output-based PI method and output-based VI method are  presented based on \cite{rizvi2019} and \cite{rizvi2023}. }

    Consider the following linear system:
    \begin{align} \label{lisys}
    	\begin{split}
    	\dot{x}=&Ax+Bu\\
    	y=&Cx,
    \end{split}
    \end{align}
    where $x\in \mathbb{R}^n$ is the system state, $u\in\mathbb{R}^m$ is the input and $y\in \mathbb{R}^p$ is the measurable output.   We  make the following assumptions:

    \begin{assmp}\label{assctrl}
    	The pair $(A,B)$ is controllable.
    \end{assmp}
    \begin{assmp}\label{ass1}
    	The pair $(A,B)$ is stabilizable.
    \end{assmp}
     \begin{assmp}\label{ass2}
    	The pair $(A,C)$ is observable.
    \end{assmp}

    The  LQR problem for \eqref{lisys} is to find a control law $u={K}x$ such that the  cost $\int_{0}^{\infty}(y^T{Q_y}y+u^T{R}u)d\tau$ is minimized,
    where ${Q_y}={Q_y}^T\geq0, {R}={R}^T>0$, with $(A,\sqrt{{Q_y}}C)$ observable.

    Let $Q=C^TQ_yC$. By \cite{kucera1972},  under Assumption \ref{ass1},   the following algebraic Riccati equation
    \begin{align}\label{areli}
    	A^T{P}^*+{P}^*A+{Q}-{P}^*B{R}^{-1}B^T{P}^*=0
    \end{align}
    admits a unique positive definite solution ${P}^*> 0$. Then, the solution of
    the LQR problem of \eqref{lisys} is given by ${K}^*=-{R}^{-1}B^T{P}^*$ and the optimal controller is $u^*=K^*x$.

   {\color{black}In what follows, we will introduce several iterative methods for approximating the optimal controller $u^*$. The super/subscript custom for denoting variables is given in TABLE \ref{tab:supersub}.}

    \begin{table}[htbp]\color{black}
    	\caption{Super/subscripts for symbols in iterative methods}
    	\centering
    	\begin{tabular}{c | >{\centering\arraybackslash}m{6cm}} 
    		\hline\hline
    		\emph{Super/subscript} & \emph{Meaning} \\
    		\hline
    	Superscript $P$ & variables for PI method
    	\\ \hline
    	Superscript $V$ & variables for VI method
    	\\ \hline
    	Subscript $k$ & Iteration step number
    	\\
    		\hline\hline
    	\end{tabular}
    	\label{tab:supersub}
    \end{table}

  \subsection{Model-Based Iterative Approaches to LQR Problem} \label{sec2-1}

 As  \eqref{areli} is nonlinear,   a model-based PI approach for obtaining ${P}^* $ is given by solving the following equations \cite{Kleinman}:
   \begin{subequations}\label{aleq1sin}
  	\begin{align}
  		0&=A_k^T{P}^P_{k} +{P}^P_{k}A_k+{Q}+({K}^P_{k})^T{R}{K}^P_{k} \\
   	{K}^P_{k+1}&=-{R}^{-1} B^T{P}^P_{k},
   \end{align}
    \end{subequations}
   where $A_k\triangleq A+B{K}^P_k$, $k=0,1,\cdots$, and ${K}^P_{0}$ is such that $A_0 $ is a Hurwitz matrix. The algorithm \eqref{aleq1sin}  guarantees the following properties for $k\in \mathbb{N}$:
   \begin{enumerate}
   	\item $\sigma(A_{k}) \subset \mathbb{C}_-$;
   	\item ${P}^*\leq {P}^P_{k+1}\leq {P}^P_k$;
   	\item $\lim\limits_{k\to \infty}{K}^P_k={K}^*,\lim\limits_{k\to \infty}{P}^P_k={P}^*$.
   \end{enumerate}

   A drawback of the model-based PI method is that it needs an initial stabilizing control gain ${K}^P_{0}$ to start the iteration. To circumvent this problem, a model-based VI approach was proposed in \cite{bian2016}, which does not need an initial stabilizing control gain. To introduce the model-based VI method in \cite{bian2016}, let $\epsilon_k$ be a series of time steps satisfying
   \begin{align}
   	\epsilon_k>0, \; \sum_{k=0}^{\infty}\epsilon_k=\infty, \; \sum_{k=0}^{\infty}\epsilon_k^2<\infty,
   \end{align}
   let    $\{B_q\}_{q=0}^{\infty}$ be a collection of bounded subsets in $\mathcal{P}^{n}$ satisfying
   \begin{align}
   	B_q\subset B_{q+1}, \; q\in \mathbb{N}, \; \lim\limits_{q\to \infty}B_q=\mathcal{P}^{n},
   \end{align}
   and let $\varepsilon>0$  be some small real number for determining the convergence criterion.
   Then the steps for obtaining the approximate solution to \eqref{areli} is shown in Algorithm \ref{vialg1}.
   \begin{algorithm}
   	\caption{{\color{black}Model-based} VI  Algorithm for solving \eqref{areli} \cite{bian2016}}\label{vialg1}
   	\begin{algorithmic}[1]
   		\State Choose ${P}^V_0=({P}^V_0)^T>0$. $k,q\leftarrow 0$.
   		\Loop
   		\State $\tilde{P}^V_{k+1}\leftarrow {P}^V_k+\epsilon_k (A^T{P}^V_k+{P}^V_kA-{P}^V_kB{R}^{-1}B^T{P}^V_k+{Q})$.
   		\If{$\tilde{P}^V_{k+1}\notin B_q$}
   		\State ${P}^V_{k+1}\leftarrow {P}^V_0$. $q\leftarrow q+1$.
   		\ElsIf{$||\tilde{P}^V_{k+1}-{P}^V_k||/\epsilon_k<\varepsilon$}
   		\State \Return ${P}^V_k$ as an approximation to ${P}^*$.
   		\Else
   		\State ${P}^V_{k+1}\leftarrow \tilde{P}^V_{k+1}$.
   		\EndIf
   		\State $k \leftarrow k+1$.
   		\EndLoop
   	\end{algorithmic}
   \end{algorithm}

It was shown in  \cite{bian2016} that, under Assumption \ref{ass1} and the assumption that  $(A,\sqrt{{Q_y}}C)$ is observable,  Algorithm \ref{vialg1} is such that $\lim\limits_{k\to \infty}{P}^V_k={P}^*$.


 \subsection{State-Based Iterative Approaches for Solving LQR Problem without Knowing $A$ and $B$} \label{sec2-2}
In this subsection, we present two iterative approaches for solving the LQR problem for \eqref{lisys} without knowing $A$ and $B$. Both of these {\color{black}methods} use only the state and input information.

{\color{black}For convenience, the data stack operators used in the following data-driven methods are summarized  in TABLE \ref{tab:symbols}, where  $a\in \mathbb{R}^{n_a}$ and $b \in \mathbb{R}^{n_b}$ are functions of time,  $u\in\mathbb{R}^m$ is the input  of \eqref{lisys}, $R$ is the performance index matrix introduced in the LQR problem of \eqref{lisys}, and $s\in \mathbb{N}_+$.}
\begin{table}[htbp]\color{black}
	\caption{Data Stack Operators}
	\centering
	\begin{tabular}{c | >{\centering\arraybackslash}m{7cm}} 
		\hline\hline
		\emph{Symbol} & \emph{Meaning} \\
		\hline
		$\delta_a$ &
		\makecell[c]{%
			$\delta_a = \Big[ \text{vecv}(a(t_1)) - \text{vecv}(a(t_0)), \cdots,$ \\
			$\text{vecv}(a(t_s)) - \text{vecv}(a(t_{s-1})) \Big]^T$
		} \\ \hline
		$\Gamma_{ab}$& \makecell[c]{$\Gamma_{ab}=[\int_{t_0}^{t_1}a\otimes b d\tau , \int_{t_1}^{t_2}a\otimes b d\tau, \cdots , \int_{t_{s-1}}^{t_s}a\otimes b d\tau]^T$}\\
		\hline
		$I_{aa}$&\makecell[c]{$I_{aa}=[\int_{t_0}^{t_1}\textup{vecv}(a)d\tau , \int_{t_1}^{t_2}\textup{vecv}(a) d\tau, \cdots ,$\\$ \int_{t_{s-1}}^{t_s}\textup{vecv}(a) d\tau]^T$}\\ \hline
		$I_{au}$&\makecell[c]{$I_{au}=[\int_{t_0}^{t_1}a\otimes {R}u d\tau , \int_{t_1}^{t_2}a\otimes {R}u d\tau, \cdots ,$\\$ \int_{t_{s-1}}^{t_s}a\otimes {R}u d\tau]^T$}\\
		\hline\hline
	\end{tabular}
	\label{tab:symbols}
\end{table}

   To introduce  the model-free algorithm proposed in \cite{jiang2012computational}  to obtain the approximate solution to \eqref{areli}, rewrite \eqref{lisys} as follows:
   \begin{align} \label{pidy}
   	\dot{x}=A_kx-B({K}^P_kx-u).
   \end{align}
   Integrating $\frac{ d (x^T (t) {{P}^P_k} x (t))}{dt}$ and using \eqref{aleq1sin} and \eqref{pidy} gives
   \begin{align}\label{piint}
   	|x(t+\delta t)|_{{P}^P_k}-&|x(t)|_{{P}^P_k}=\int_{t}^{t+\delta t}-|x|_{{Q}+({K}^P_{k})^T{R}{K}^P_{k}}\notag \\
   	&+2x^T({K}^P_k)^T{R}{K}^P_{k+1}x-2u^T{R}{K}^P_{k+1}x d\tau.
   \end{align}


   Thus, \eqref{piint}  implies
   \begin{align}\label{pilinear}
   	\Psi_k^P\begin{bmatrix}
   		\text{vecs}({P}^P_k)\\
   		\text{vec}({K}^P_{k+1})
   	\end{bmatrix}=\Phi_k^P,
   \end{align}
   where $\Psi_k^P=[\delta_x, -2\Gamma_{xx}(I_n\otimes (({K}^P_k)^T{R}))+2\Gamma_{xu}(I_n\otimes {R})]$ and $\Phi_k^P=-\Gamma_{xx}\text{vec}({Q}+({K}^P_{k})^T{R}{K}^P_{k})$.

   The solvability of \eqref{pilinear} is guaranteed by the following lemma \cite{jiang2012computational}.
   \begin{lem}
   	The matrix $\Psi_k^P$ has full column rank for {\color{black}$k\in\mathbb{N}$}, if
   	\begin{align}\label{rankconpi}
   		\textup{rank}([\Gamma_{xx}, \Gamma_{xu}])=& \frac{n(n+1)}{2}+mn.
   	\end{align}
   \end{lem}

   Thus, by iteratively solving \eqref{pilinear}, one can finally obtain the approximated solution to \eqref{areli} without using $A$ and $B$.

    The iteration approach using \eqref{pilinear} is named the state-based PI method. Like the model-based PI method,  it needs an initial stabilizing control gain ${K}^P_{0}$.  When the matrices $A$ and $B$ are unknown, it is difficult to obtain such a stabilizing ${K}^P_{0}$. To overcome this difficulty, \cite{bian2016} developed a model-free state-based VI method as follows.

     Let ${H}^V_k=A^T{P}^V_k+{P}^V_kA, {K}^V_k=-{R}^{-1}B^T{P}^V_k$.
    Then,
    it is obtained from \eqref{lisys} that
    \begin{align}\label{intvipre}
    	|x(t+\delta t)|_{{P}^V_k}-|x(t)|_{{P}^V_k} =\int_{t}^{t+\delta t}\lbrack |x|_{{H}^V_k}-2{u}^T{R}{K}^V_kx \rbrack d\tau.
    \end{align}

  \eqref{intvipre}  implies
   \begin{align}\label{vilinearpre}
   	\Psi^V\begin{bmatrix}
   		\text{vecs}({H}^V_k)\\
   		\text{vec}({K}^V_k)
   	\end{bmatrix}=\Phi^V_k,
   \end{align}
   where $\Psi^V=\begin{bmatrix}
   	I_{xx}&-2I_{xu}
   \end{bmatrix}$ and $\Phi^V_k=\delta_{x}\text{vecs}({P}^V_k)$.

The solvability of \eqref{vilinearpre} is guaranteed by the following lemma \cite{bian2016}.
\begin{lem}
	The matrix $\Psi^V$ has full column rank  if
	\begin{align}\label{rankconvi}
		\textup{rank}([I_{xx}, I_{xu}])=& \frac{n(n+1)}{2}+mn.
	\end{align}
\end{lem}

  {\color{black} \begin{rem}
   	In Step 3 of Algorithm \ref{vialg1}, the updating law of $\tilde{P}^V_{k+1}$ can be represented by $\tilde{P}^V_{k+1}\leftarrow {P}^V_k+\epsilon_k ({H}^V_k-({K}^V_k)^T{R}{K}^V_k+{Q})$. Thus, once ${H}^V_k$ and ${K}^V_k$ are solved from \eqref{vilinearpre}, one can update ${P}^V_{k+1}$ following Step 2-12 of Algorithm \ref{vialg1}.
   \end{rem}}

   Thus, by iteratively applying the solution of \eqref{vilinearpre} to Algorithm \ref{vialg1}, we can finally obtain the approximated solution to \eqref{areli} without knowing $A$ and $B$.

    \subsection{Output-Based Iterative Approaches for Solving LQR Problem {\color{black}Without} Knowing $A$, $B$ and $C$}\label{sec2-3}
    In most applications, the internal state is not available. \cite{rizvi2019} and \cite{rizvi2023} proposed both PI and VI data-driven methods for designing output feedback control law  based on the state parametrization technique.  In this subsection, we will summarize these two methods based on  \cite{rizvi2019} and \cite{rizvi2023}.

    Under Assumption \ref{ass2}, there exists an observer gain $L$ such that the eigenvalues of $A-LC$  can be arbitrarily placed. If $A-LC$ is Hurwitz, then
    the following Luenberger observer:
    \begin{align}\label{observer}
    	\dot{\hat{x}}= &A\hat{x}+Bu+L(y-C\hat{x})\notag \\
    	=&(A-LC)\hat{x}+Bu+Ly
    \end{align}
    will drive the state $\hat{x}$ to $x$ exponentially{\color{black}\cite{luenberger}}.

Let   $\Lambda(s)=\textup{det}(sI-A+LC) = s^n + \alpha_{n-1} s^{n-1} + \cdots+ \alpha_1 s + \alpha_0$. Then $(sI-A+LC)^{-1}=\frac{D_{n-1}s^{n-1}+D_{n-2}s^{n-2}+\cdots+D_1s+D_0}{\Lambda(s) }$ where $D_i\in \mathbb{R}^{n\times n}, i=0,1,\cdots, n-1$.
Let   $\zeta^i(t)\in \mathbb{R}^n$ be governed  by the following  system:
\begin{align} \label{observer2}
	\dot{\zeta}^i (t)=\mathcal{A} \zeta^i (t) +b w_i(t), ~~i = 1, \cdots, m+p,
\end{align}
where $w_i (t) = u_i (t)$, $i = 1, \cdots, m$,  $w_i (t) = y_{i-m}(t)$, $i = m+1, \cdots, m+p$, and
 \begin{align*}
	\mathcal{A}=\begin{bmatrix}
		0&1&0& \cdots& 0\\
		0& 0& 1& \cdots & 0\\
		\vdots & \vdots&\vdots& \vdots& \vdots\\
		0 & 0& 0& \cdots& 1\\
		-\alpha_0 & -\alpha_1& -\alpha_2& \cdots& -\alpha_{n-1}
	\end{bmatrix},~ b=\begin{bmatrix}
		0\\0\\ \vdots \\0\\ 1
	\end{bmatrix}.
\end{align*}

 \begin{rem}
 Since $\Lambda(s)$ is user-defined, $\mathcal{A}$ and $b$ are known matrices.
\end{rem}

Let $\zeta = \col (\zeta^1, \cdots,  \zeta^{p+m} )$ and  $M = [M_1,\cdots ,M_{p+m}]$ where $M_i =  [D_0 f_i ,\cdots ,D_{n-1} f_i]$ with $f_i$ the $i^{th}$ column of $B$ for $i =1, \cdots, m$ and $(i - m)^{th}$ column of $L$ for $i = m+1, \cdots, m+p$. Since $M\zeta (s) = (sI - A+ LC)^{-1} [B~~ L]\begin{bmatrix}
	u(s)\\y(s)
\end{bmatrix} $ where $\zeta (s),u(s),y(s)$ are the Laplace transform of $\zeta (t),u(t),y(t)$, we have
   $\lim_{t \rightarrow \infty} ( M \zeta (t) - x (t)) = 0$ exponentially. {\color{black}Thus, the optimal controller $u^*=K^*x$ can be approximated by $u=K^*M\zeta$. The idea of this output-based optimal controller design method is summarized in Fig. \ref{diagram}. The next question is how to compute the control gain $K^*M$ from data. For this purpose, \cite{rizvi2019} and \cite{rizvi2023} proposed both output-based PI and VI methods as follows.}

    \begin{figure}[H]\color{black}
   	\centering
   	\includegraphics[width=8.5cm]{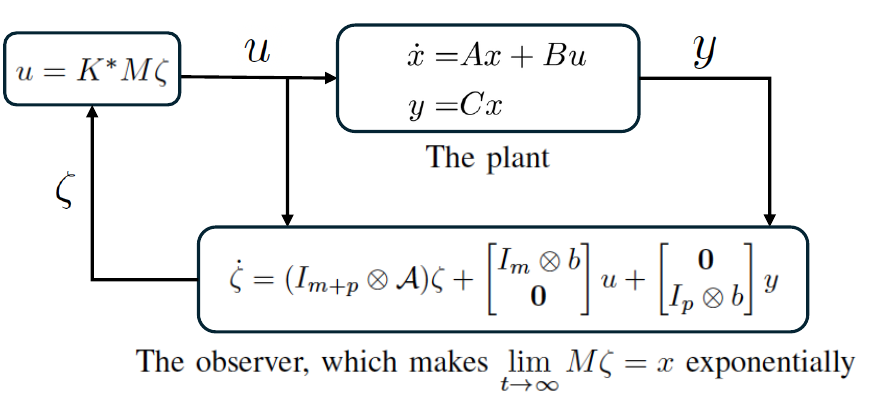}
   	\caption{Output-based optimal controller design method.}
   	\label{diagram}
   \end{figure}

    Let $\bar{P}_k^P =  M^T{P}_k^PM$ and $\bar{K}_k^P = {K}_k^PM$.
Then  replacing $x(t)$ in \eqref{piint} by $M\zeta(t)$ for $t\geq t_0$ with a sufficiently large $t_0>0$ gives:
    \begin{align}\label{piintout}
    	|\zeta(t+\delta t)|_{\bar{P}_k^P}-&|\zeta(t)|_{\bar{P}_k^P}=\int_{t}^{t+\delta t}-|y|_{Q_y}-\zeta^T(\bar{K}_k^P)^TR\bar{K}_k^P\zeta\notag \\
    	&+2\zeta^T(\bar{K}^P_k)^T{R}\bar{K}^P_{k+1}\zeta-2u^T{R}\bar{K}^P_{k+1}\zeta d\tau.
    \end{align}

    Thus, \eqref{piintout} implies
    \begin{align}\label{pilinearout}
    	\bar{\Psi}_k^P\begin{bmatrix}
    		\text{vecs}(\bar{P}^P_k)\\
    		\text{vec}(\bar{K}^P_{k+1})
    	\end{bmatrix}=\bar{\Phi}_k^P,
    \end{align}
    where $\bar{\Psi}_k^P=[\delta_\zeta, -2\Gamma_{\zeta\zeta}(I_{n_\zeta}\otimes ((\bar{K}^P_k)^T{R}))+2\Gamma_{\zeta u}(I_{n_\zeta}\otimes {R})]$, $\bar{\Phi}_k^P=-\Gamma_{yy}\text{vec}(Q_y)-\Gamma_{\zeta\zeta}\text{vec}((\bar{K}^P_{k})^T{R}\bar{K}^P_{k}), n_\zeta=(m+p)n$.

    The solvability of \eqref{pilinearout} is guaranteed by the following lemma \cite{rizvi2023}.
    \begin{lem}
    	The matrix $\bar{\Psi}_k^P$ has full column rank for {\color{black}$k\in\mathbb{N}$}, if
    	\begin{align}\label{rankconpiout}
    		\textup{rank}([\Gamma_{\zeta\zeta}, \Gamma_{\zeta u}])=& \frac{n_\zeta(n_\zeta+1)}{2}+mn_\zeta,
    	\end{align}
    	where $n_\zeta=(m+p)n$.
    \end{lem}


%

   The iterative method using \eqref{pilinearout} is called the output-based PI method, which is summarized as Algorithm \ref{algpiout}.
   \begin{algorithm}
   	\caption{The output-based PI Algorithm \cite{rizvi2019}\cite{rizvi2023}}\label{algpiout}
   	\begin{algorithmic}[1]
   		\State {\color{black}Find a feedback gain $\bar{K}^P_{0}$ such that $\bar{K}^P_0\triangleq {K}_0^PM$ with $A+B{K}^P_0$ Hurwitz.} Apply an initial input $u^0=\bar{K}^P_{0}\zeta+\delta$ with exploration noise $\delta$. Choose a proper $t_0>0$ such that $||x(t)-M\zeta(t)||$ is small enough.
   		\State {\color{black}Collect} data from $t_0$ until the rank condition \eqref{rankconpiout} is satisfied. $k\leftarrow 0$.
   		\Repeat
   		\State Solve $\bar{P}^P_k,\bar{K}^P_{k+1}$ from (\ref{pilinearout}). $k\leftarrow k+1$.
   		\Until{$||\bar{P}^P_k-\bar{P}^P_{k-1}||<\varepsilon$ with a sufficiently small constant $\varepsilon>0$}.
   		\State $k^*\leftarrow k$.
   		\State Obtain the following optimal controller
   		\begin{align}\label{controllerpipre}
   			u^*=\bar{K}^P_{k^*}\zeta.
   		\end{align}
   	\end{algorithmic}
   \end{algorithm}

  The output-based PI method based on \eqref{pilinearout} requires a  stable initial policy $\bar{K}^P_0\triangleq {K}_0^PM$ with $A+B{K}^P_0$ Hurwitz  to start iteration. However, without the information of $A,B,C$, it is difficult to get such a stable initial policy $\bar{K}^P_0$ in advance. Moreover, even if a stabilizing ${K}_0^P$  is available, one still cannot get a  $\bar{K}^P_0$ 	since   $M$ is also unknown.
 To  overcome this difficulty,    \cite{rizvi2019} and \cite{rizvi2023} also proposed an output-based VI algorithm as follows.

   Let  $\bar{P}^V_k=M^T{P}^V_k M, \bar{H}^V_k=M^T(A^T{P}^V_k+{P}^V_kA+Q)M, \bar{K}^V_k=-{R}^{-1}B^T{P}^V_kM$. Then,  replacing $x(t)$ by $M\zeta(t)$ in \eqref{intvipre} for $t \geq t_0$ for some sufficiently large $t_0$ gives

   \begin{align}\label{intviout}
   	|\zeta(t+\delta t)|_{\bar{P}^V_k}-&|\zeta(t)|_{\bar{P}^V_k} +\int_{t}^{t+\delta t} |y|_{Q_y} d\tau\notag \\=&\int_{t}^{t+\delta t}\lbrack |\zeta|_{\bar{H}^V_k}-2{u}^T{R}\bar{K}^V_k\zeta \rbrack d\tau.
   \end{align}

   \eqref{intviout}  implies
   \begin{align}\label{vilinearout}
   	\bar{\Psi}^V\begin{bmatrix}
   		\text{vecs}(\bar{H}^V_k)\\
   		\text{vec}(\bar{K}^V_k)
   	\end{bmatrix}=\bar{\Phi}^V_k,
   \end{align}
   where $\bar{\Psi}^V=\begin{bmatrix}
   	I_{\zeta\zeta}&-2I_{\zeta u}
   \end{bmatrix}$ and $\bar{\Phi}^V_k=\delta_{\zeta}\text{vecs}(\bar{P}^V_k)+I_{yy}\text{vecs}(Q_y)$.

   The solvability of \eqref{vilinearout} is guaranteed by the following lemma \cite{rizvi2023}.
   \begin{lem}
   	The matrix $\bar{\Psi}^V$ has full column rank  if
   	\begin{align}\label{rankconviout}
   		\textup{rank}([I_{\zeta \zeta}, I_{\zeta u}])=& \frac{n_\zeta(n_\zeta+1)}{2}+mn_\zeta.
   	\end{align}
    \end{lem}

  The above so-called output-based VI algorithm from  \cite{rizvi2019} and \cite{rizvi2023} is summarized as Algorithm \ref{vialgout}, where $\epsilon_k$ and $\varepsilon$ are defined in the same way as those in Section \ref{sec2-2},  $\{\bar{B}_q\}_{q=0}^{\infty}$ is a collection of bounded subsets in $\mathcal{P}^{n_\zeta}$ satisfying
  \begin{align}
  	\bar{B}_q\subset \bar{B}_{q+1}, \; q\in \mathbb{N}, \; \lim\limits_{q\to \infty}\bar{B}_q=\mathcal{P}^{n_\zeta}.
  \end{align}

  \begin{algorithm}
  	\caption{{\color{black}Output-based} VI  Algorithm  \cite{rizvi2019}\cite{rizvi2023}}\label{vialgout}
  	\begin{algorithmic}[1]
  		\State  Apply any locally essentially bounded initial input $u^0$. Choose a proper $t_0>0$ such that $||x(t)-M\zeta(t)||$ is small enough.
  		{\color{black}Collect} data starting from $t_0$ until the rank condition \eqref{rankconviout} is satisfied.
  		\State Choose $\bar{P}^V_0=(\bar{P}^V_0)^T>0$. $k,q\leftarrow 0$.
  		\Loop
  		\State Solve $\bar{H}^V_k$ and $\bar{K}^V_k$ from \eqref{vilinearout}.
  		\State $\tilde{\bar{P}}^V_{k+1}\leftarrow \bar{P}^V_k+\epsilon_k (\bar{H}^V_k-(\bar{K}^V_k)^TR\bar{K}^V_k)$.
  		\If{$\tilde{\bar{P}}^V_{k+1}\notin \bar{B}_q$}
  		\State $\bar{P}^V_{k+1}\leftarrow \bar{P}^V_0$. $q\leftarrow q+1$.
  		\ElsIf{$||\tilde{\bar{P}}^V_{k+1}-\bar{P}^V_k||/\epsilon_k<\varepsilon$}
  		\State \Return $\bar{P}^V_k$ and $\bar{K}^V_k$ as  approximations to $\bar{P}^*\triangleq M^TP^*M$ and $\bar{K}^*\triangleq K^*M$, respectively.
  		\Else
  		\State $\bar{P}^V_{k+1}\leftarrow \tilde{\bar{P}}^V_{k+1}$.
  		\EndIf
  		\State $k \leftarrow k+1$.
  		\EndLoop
  	\end{algorithmic}
  \end{algorithm}

  The following theorem summarized from Theorem 2.13 in  \cite{rizvi2023} {\color{black}ensures} the convergence of Algorithm \ref{algpiout} and Algorithm \ref{vialgout}.

  \begin{thm}\label{thm3}
  	Under Assumptions \ref{assctrl} and \ref{ass2}, let $(A,\sqrt{{Q_y}}C)$ be observable, if $M$ has full row rank and the rank condition \eqref{rankconpiout} and \eqref{rankconviout} are satisfied, $\bar{P}^P_k, \bar{K}^P_{k}$ generated from Algorithm \ref{algpiout} and $\bar{P}^V_k, \bar{K}^V_{k}$ generated from Algorithm \ref{vialgout} are such that
  	\begin{align*}
  			\lim\limits_{k\to \infty}\bar{P}^P_k=&M^TP^*M\\
  		\lim\limits_{k\to \infty}\bar{K}^P_k=&K^*M\\
  		\lim\limits_{k\to \infty}\bar{P}^V_k=&M^TP^*M\\
  		\lim\limits_{k\to \infty}\bar{K}^V_k=&K^*M.
  	\end{align*}
  \end{thm}


    \subsection{ A New Result}
   It is noted that the Theorem 2.13 of \cite{rizvi2023} requires the linear system \eqref{lisys} to be controllable, and  the $M$ matrix to have full row rank.
   In this subsection, we will show that Assumption \ref{assctrl} implies that  $M$ is of full row rank.  Moreover, in the next section, we will develop a new approach
   that  does not require  the full row rank condition of $M$, and the controllable condition Assumption \ref{assctrl} can be extended to the stabilizable condition Assumption \ref{ass1}.

   For this purpose, note that
  by \eqref{observer2} and the definition of $\zeta$,   we have
  \begin{align}\label{dyz}
  	\dot{\zeta}=
  	(I_{m+p}\otimes \mathcal{A})
  	\zeta+\begin{bmatrix}
  		I_m\otimes b\\ {\bf0}
  	\end{bmatrix}u+\begin{bmatrix}
  		{\bf0}\\I_p\otimes b
  	\end{bmatrix}y.
  \end{align}


  For further discussion, we give the following useful relationships.
  \begin{lem}\label{lemrelation}
  Let $\Lambda(s) = \textup{det}(sI-A+LC)$. Then,
  	\begin{subequations}
  		\begin{align}\label{rela1}
  		M 	(I_{m+p}\otimes \mathcal{A}) =&(A-LC)M\\ \label{rela2}
  		M\begin{bmatrix}
  			I_m\otimes b\\ {\bf0}
  		\end{bmatrix}=&B\\ \label{rela3}
  		M\begin{bmatrix}
  			{\bf0}\\I_p\otimes b
  		\end{bmatrix}=&L.
  	\end{align}
  	\end{subequations}
  \end{lem}
  \begin{proof}
  	It is well known that, for example, see  Problem 3.26 in \cite{ctchenbook},  the following  relationship holds:
  {\color{black}	\begin{align}\label{recursion}
  		\begin{split}
  			D_{n-1}=&I\\
  			D_{n-2}=&(A-LC)D_{n-1}+\alpha_{n-1}I\\
  			\vdots \\
  			D_0=&(A-LC)D_{1}+\alpha_{1}I\\
  			0=&(A-LC)D_{0}+\alpha_{0}I.
  		\end{split}
  	\end{align}
  }
  	Post-multiplying $f_i$ on both sides of \eqref{recursion} gives
  {\color{black}	\begin{align*}
  		\begin{split}
  			D_{n-1}f_i=&f_i\\
  			D_{n-2}f_i-\alpha_{n-1}f_i=&(A-LC)D_{n-1}f_i\\
  			\vdots \\
  			D_0f_i-\alpha_{1}f_i=&(A-LC)D_{1}f_i\\
  			-\alpha_{0}f_i=&(A-LC)D_{0}f_i.
  		\end{split}
  	\end{align*}
  }
  	Thus,
  	\begin{align*}
  		&M_i\mathcal{A}\\=&[-\alpha_{0}D_{n-1}f_i, -\alpha_{1}D_{n-1}f_i+D_0f_i, \cdots,\\&-\alpha_{n-1}D_{n-1}f_i+D_{n-2}f_i]\\
  		=&[-\alpha_{0}f_i, D_0f_i-\alpha_{1}f_i, \cdots,D_{n-2}f_i-\alpha_{n-1}f_i]\\
  		=&[(A-LC)D_{0}f_i, (A-LC)D_{1}f_i, \cdots,(A-LC)D_{n-1}f_i]\\
  		=&(A-LC)M_i, i=1,2,\cdots,m+p,
  	\end{align*}
  	which implies
  	\begin{align*}
  		M (I_{m+p}\otimes \mathcal{A})=&[M_1\mathcal{A}, M_2\mathcal{A},\cdots,M_{m+p}\mathcal{A}]\notag \\
  		=&(A-LC)M.
  	\end{align*}
  	
  	Thus, \eqref{rela1} holds.
  	
  	Since $M_i b = D_{n-1} f_i$ for $i = 1, \cdots, m+p$, we have
  	\begin{align*}
  		M	(I_{m+p}\otimes b)
  		=&[M_1b,M_2b,\cdots,M_{m+p}b]\\
  		=&[f_1, f_2, \cdots, f_{m+p}] = [B~~ L].
  	\end{align*}
  	Thus, \eqref{rela2} and \eqref{rela3} hold.
  \end{proof}

  {\color{black}
  	Based on  Lemma \ref{lemrelation}, we have the following result.
  	\begin{prop}\label{prop1}
  		Under Assumption \ref{assctrl},  $M$ has full row rank.
  	\end{prop}
  	\begin{proof}
  		Let $M'=\begin{bmatrix}
  			D_0&D_1&\cdots&D_{n-1}
  		\end{bmatrix}(I_{n}\otimes \begin{bmatrix}
  			B& L
  		\end{bmatrix})$. Then, $\textup{rank}(M')=\textup{rank}(M)$ since $M'$ can be obtained by rearranging the columns of $M$. Thus, $M$ has full row rank if and only if $M'$ has full row rank.
  		
  		From  \eqref{recursion}, we have
  		\begin{align}\label{recursion2}
  			\begin{split}
  				D_{n-1}=&I\\
  				D_{n-2}=&(A-LC)+\alpha_{n-1}I\\
  				D_{n-3}=&(A-LC)^2+\alpha_{n-1}(A-LC)+\alpha_{n-2}I\\
  				\vdots \\
  				D_0=&(A-LC)^{n-1}+\alpha_{n-1}(A-LC)^{n-2}+\cdots+\alpha_{1}I.
  			\end{split}
  		\end{align}
  		Let
  	\begin{small}
  			\begin{align*}
  			T_c=\begin{bmatrix}
  				(A-LC)^{n-1}[B ~ L] & (A-LC)^{n-2}[B ~ L]& \cdots & [B ~ L]
  			\end{bmatrix}.
  		\end{align*}
  	\end{small}
  		Then, $M'$ can be expressed as
  		\begin{align*}
  			M'=&T_c \left (\begin{bmatrix}
  				1&0&0&\cdots&0\\
  				\alpha_{n-1}&1&0&\cdots&0\\
  				\alpha_{n-2}&\alpha_{n-1}&1&\cdots&0\\
  				\vdots&\vdots&\vdots&\cdots&\vdots\\
  				\alpha_{1}&\alpha_{2}&\alpha_{3}&\cdots&1
  			\end{bmatrix}\otimes I_{m+p} \right).
  		\end{align*}
  		Thus,  $M'$ has full row rank if and only if $T_c$ has full row rank, {\color{black}which is equivalent to  $(A-LC, [B~L])$ is controllable.} Thus, it suffices  to prove that Assumption \ref{assctrl} implies  the controllability of  $(A-LC, [B~L])$ .
  		
  		By the  PBH test,  Assumption \ref{assctrl} implies  $\textup{rank}([\lambda I_n-A,B])=n, \forall \lambda\in \mathbb{C}$. Since
  		\begin{align*}
  			&\begin{bmatrix}
  				\lambda I_n-A+LC& B&L
  			\end{bmatrix}\\=&\begin{bmatrix}
  				\lambda I_n-A&B&L
  			\end{bmatrix}\begin{bmatrix}
  				I_n &{\bf0}&{\bf0}\\ {\bf0}&I_m&{\bf0}\\ C&{\bf0}& I_p
  			\end{bmatrix}, \forall \lambda\in \mathbb{C},
  		\end{align*}
  		we have $\textup{rank}([\lambda I_n-A+LC,B,L])=\textup{rank}([\lambda I_n-A,B,L])$, $\forall \lambda\in \mathbb{C}$. Thus, Assumption \ref{assctrl}  implies  $\textup{rank}([\lambda I_n-A,B,L])=n, \forall \lambda\in \mathbb{C}$,  or what is the same,  the controllability of  $(A-LC, [B~L])$.
  		
  		In conclusion, Assumption \ref{assctrl} implies $M$ is of full row rank.
  		
  	\end{proof}
  	
  	\begin{rem}
  	It was shown in  Theorem 2.10 and Theorem 2.11 of \cite{rizvi2023} that $M$ has full row rank if $(A-LC,B)$ or $(A-LC,L)$ is controllable, or $A$ and $A-LC$ have no common eigenvalues. In contrast,  Proposition \ref{prop1} concludes that controllability of  $(A,B)$  is enough to guarantee that $M$ is of  full row rank.
  	\end{rem}

   \section{A New Approach to the Output-Based ADP Algorithms}\label{sec3}
 {\color{black}We have seen that} the approach of \cite{rizvi2019} and  \cite{rizvi2023} requires the matrix $M$ to be
  of full row rank, and needs to solve a sequence of linear algebraic equations with $\frac{n_\zeta(n_\zeta+1)}{2} + m n_\zeta$ unknown variables.
  Even though Proposition  \ref{prop1} showed that $M$ is of full row rank if the plant is controllable, it is still desirable that the approach also applies to stabilizable systems.
 In this section, we will present a new approach that imposes no requirement on the rank of $M$ and that only needs to
  solve a sequence of linear algebraic equations with  $\frac{n_\zeta(n_\zeta+1)}{2}$ unknown variables,
  thus reducing the number of unknown variables of the existing algorithms by $m n_\zeta$. Moreover,
  the existing  PI or VI algorithm only applies to the case where  some Riccati equation admits a unique positive definite solution. However,
  when it comes to the
  output-based PI or VI algorithm, one can only guarantee that
  the solution of the Riccati equation is positive semi-definite. To fill this gap, we further show that the PI algorithm as given in \cite{Kleinman}  and the VI algorithm as given in
  \cite{bian2016} can be generalized to the case where  the Riccati equation only admits a unique positive semi-definite solution.

  \subsection{An Ancillary System and Its LQR Problem}\label{sec3-0}
	Let $e_\zeta\triangleq M\zeta(t)-x$. Then
    substituting $y=Cx=CM\zeta-Ce_\zeta$ into   \eqref{dyz} gives
   \begin{align}\label{dyz2bf}
     	\dot{\zeta}=A_\zeta\zeta+B_\zeta u+d_\zeta,
     \end{align}
     where
     \begin{align*}
     	A_\zeta&=
     		(I_{m+p}\otimes \mathcal{A})+\begin{bmatrix}
     	{\bf0}\\I_p\otimes b
     	\end{bmatrix}CM\\
     	B_\zeta&=\begin{bmatrix}
     		I_m\otimes b\\ {\bf0}
     	\end{bmatrix},d_\zeta=-\begin{bmatrix}
     	{\bf0}\\I_p\otimes b
     	\end{bmatrix}Ce_\zeta.
     \end{align*}

    What makes \eqref{dyz2bf} interesting is that $A_\zeta$ is an unknown matrix since $CM$ is unknown, but  $B_\zeta$ is known.
    We will first {\color{black}prove} that
     $(A_\zeta,B_\zeta)$ is  stabilizable.
     \begin{lem}\label{lem6}
     	Under Assumption \ref{ass1} and \ref{ass2}, let $\Lambda(s)=\textup{det}(sI-A+LC)$ be Hurwitz. Then, the pair $(A_\zeta, B_\zeta)$  is stabilizable.
     \end{lem}
     \begin{proof}
     	With \eqref{lisys}, \eqref{dyz} and Lemma \ref{lemrelation}, the dynamics of $e_\zeta\triangleq M\zeta(t)-x$ is derived as follows
     	\begin{align}
     	\begin{split}
     			\dot{e}_\zeta=&M\dot{\zeta}-\dot{x}\\
     			=& M((I_{m+p}\otimes \mathcal{A})
     			\zeta+\begin{bmatrix}
     				I_m\otimes b\\ {\bf0}
     			\end{bmatrix}u+\begin{bmatrix}
     				{\bf0}\\I_p\otimes b
     			\end{bmatrix}y)\\
     			&-Ax-Bu\\
     			=&(A-LC)M\zeta+Bu+Ly-Ax-Bu\\
     			=&(A-LC)e_\zeta.
     	\end{split}
     	\end{align}
     	Since $A-LC$ is Hurwitz, $e_\zeta$ is exponentially stable.
     	
     	On the other hand, since the pair $(A,B)$ is stabilizable, there exists a control gain $K$ such that $A+BK$ is Hurwitz.  Applying $u=KM\zeta$ to \eqref{lisys} gives
     	\begin{align*}
     		\dot{x}=&Ax+Bu\\
     		=& Ax+BKM\zeta\\
     		=& Ax+BK(x+e_\zeta)\\
     		=& (A+BK) x+BKe_\zeta.
     	\end{align*}
     Since $e_\zeta$ tends to 0 exponentially,  by Lemma 1 of \cite{cai2017}, $x$ converges to 0 exponentially since $A+BK$ is Hurwitz.
     	
     	
     	Now, consider the following linear system
     \begin{align}
     	\begin{split}\label{auglisysez}
     		\dot{\zeta}=&A_\zeta\zeta+B_\zeta u+d_\zeta\\
     		\dot{e}_\zeta=&(A-LC)e_\zeta.
     	\end{split}
     \end{align}
     Since $d_\zeta=-\begin{bmatrix}
     	{\bf0}\\I_p\otimes b
     \end{bmatrix}Ce_\zeta$, applying $u=K_\zeta\zeta$ with $K_\zeta=KM$ to \eqref{auglisysez} gives
     \begin{align}\label{auglisys}
     	\begin{bmatrix}
     		\dot{\zeta}\\\dot{e}_\zeta
     	\end{bmatrix}=\begin{bmatrix}
     		A_\zeta+B_\zeta K_\zeta&  -\begin{bmatrix}
     			{\bf0}\\I_p\otimes b
     		\end{bmatrix}C\\ {\bf 0} & A-LC
     	\end{bmatrix}\begin{bmatrix}
     		\zeta\\e_\zeta
     	\end{bmatrix}.
     \end{align}
     We now further show that $\zeta$ tends to 0 exponentially. In fact,  applying $u=KM\zeta = K (x+e_\zeta)$ to \eqref{dyz2bf} gives
     	\begin{align}\label{dyz2bf2}
     		\begin{split}
     				\dot{\zeta}
     			=&(I_{m+p}\otimes \mathcal{A})\zeta+\begin{bmatrix}
     			{\bf0}\\I_p\otimes b
     		\end{bmatrix}CM\zeta \\
     		&+B_\zeta K M\zeta  - \begin{bmatrix}
     			{\bf0} \\ I_p\otimes b
     		\end{bmatrix} C e_\zeta  \\
     			=&(I_{m+p}\otimes \mathcal{A})\zeta+ \left  (\begin{bmatrix}
     			{\bf0}\\I_p\otimes b
     		\end{bmatrix}C +B_\zeta K \right ) x \\
     		&+  B_\zeta K   e_\zeta  	.
     	\end{split}
     	\end{align}
     	Since both $x$ and $e_\zeta$ converge to 0 exponentially and $(I_{m+p} \otimes \mathcal{A})$  is Hurwitz, the solution $\zeta$ of \eqref{dyz2bf2} goes to 0 exponentially by  Lemma 1 of \cite{cai2017} again.

     	Since both $\zeta$ and $e_\zeta$ converge to 0 exponentially, the linear system \eqref{auglisys} is exponentially stable. Thus,
     	the matrix $	A_\zeta+B_\zeta K_\zeta$ must be  Hurwitz, implying that the pair $(A_\zeta, B_\zeta)$ is stabilizable.
     \end{proof}

     Since $d_\zeta$ tends to 0 exponentially, the effect of $d_\zeta$ on \eqref{dyz2bf} can be ignored after some finite time $t_0>0$. As a result, \eqref{dyz2bf} can be simplified to the following form:
     \begin{align}\label{dyz2}
     	\dot{\zeta}=A_\zeta\zeta+B_\zeta u, t\geq t_0.
     \end{align}

     Next, we further establish the following detectability result.

     \begin{lem}\label{lem7}
     	Under Assumption \ref{ass2}, let  $\Lambda(s)=\textup{det}(sI-A+LC)$ be stable. Then, for any $Q_y>0$, the pair $(A_\zeta, \sqrt{{Q_y}}CM)$ is detectable.
     \end{lem}
     \begin{proof}
     Since $Q_y>0$, $\sqrt{{Q_y}}$ is an invertible matrix. Let $L_\zeta=\begin{bmatrix}
     		{\bf0}\\I_p\otimes b
     	\end{bmatrix}\sqrt{{Q_y}}^{-1}$. Then, we have
     	\begin{align*}
     		A_\zeta-L_\zeta\sqrt{{Q_y}}CM=&
     			I_{m+p} \otimes \mathcal{A}.
     	\end{align*}
     	
     	Since $\mathcal{A}$ is Hurwitz, $A_\zeta-L_\zeta\sqrt{{Q_y}}CM$ is Hurwitz. In conclusion, the pair $(A_\zeta, \sqrt{{Q_y}}CM)$ is detectable.
     \end{proof}

       Now, consider the problem  of finding a control law $u=u_\zeta={K}_\zeta \zeta$ to minimize the  cost $\int_{t_0}^{\infty}(\zeta^T{Q_\zeta}\zeta+u^T{R}u)d\tau$,
     where ${Q_\zeta}=M^TQM=M^TC^TQ_yCM\geq0$, $ Q_y$ and ${R}$ are the same as those of \eqref{areli}.
     Since $(A_\zeta,B_\zeta)$ is stabilizable and $(A_\zeta, \sqrt{{Q_y}}CM)$ is detectable, then, an optimal control gain for this LQR problem of \eqref{dyz2} is given by $K_\zeta^*=-{R}^{-1}B_\zeta^T{P}_\zeta^*$ where ${P}_\zeta^*$ is the unique positive semi-definite solution to the following algebraic Riccati equation \eqref{areliz}:
     \begin{align}\label{areliz}
     		A_\zeta^T{P}_\zeta^*+{P}_\zeta^*A_\zeta+Q_\zeta-{P}_\zeta^*B_\zeta{R}^{-1}B_\zeta^T{P}_\zeta^*=0.
     	\end{align}
    In fact, the relation between ${P}_\zeta^*$ and $P^*$, the unique positive definite solution to \eqref{areli}, is given as follows:

     \begin{thm}\label{thm4}
     	Under Assumptions \ref{ass1} and \ref{ass2}, let $L$ be such that $\Lambda(s)=\textup{det}(sI-A+LC)$ is Hurwitz. Then, if $Q_y>0$, the algebraic Riccati equation \eqref{areliz}
     	admits a unique positive semi-definite solution ${P}_\zeta^*=M^TP^*M\geq 0$, where $P^*$ is the unique positive definite solution to \eqref{areli}.
     \end{thm}
     \begin{proof}
     	By Lemma \ref{lem6} and \ref{lem7}, the pair $(A_\zeta, B_\zeta)$ is stabilizable and the pair $(A_\zeta, \sqrt{{Q_y}}CM)$ is detectable. Thus, based on \cite{kucera1972}, \eqref{areliz} admits a unique positive semi-definite solution ${P}_\zeta^*$.
     	
     	On the other hand, since $(A,C)$ is observable and $Q_y>0$, $(A,\sqrt{{Q_y}}C)$ is observable. Thus, together with Assumption \ref{ass1}, the algebraic Riccati equation \eqref{areli}  admits a unique positive definite solution ${P}^*$.
     	
     	Now, we verify that the positive semidefinite matrix $M^TP^*M$ is the solution to \eqref{areliz}.
     	
     	From \eqref{rela1} and \eqref{rela3}, we have
     	\begin{align*}
     		MA_\zeta=&M (I_{m+p} \otimes \mathcal{A}+\begin{bmatrix}
     			{\bf0}\\I_p\otimes b
     		\end{bmatrix}CM)\\
     		=&(A-LC)M+LCM\\
     		=& AM.
     	\end{align*}
     	Thus, we obtain
     	\begin{align}\label{PA}
     		\begin{split}
     			(M^TP^*M)A_\zeta=&M^TP^*AM\\
     			A_\zeta^T(M^TP^*M)=&M^TA^TP^*M.
     		\end{split}
     	\end{align}
     	
     	Moreover, from \eqref{rela2}, we have $M^TP^*MB_\zeta=M^TP^*B$. Thus,
     	\begin{align}\label{PB}
     		\begin{split}
     			(M^TP^*M)B_\zeta{R}^{-1}B_\zeta^T(M^TP^*M)=M^TP^*BR^{-1}B^TP^*M.
     		\end{split}
     	\end{align}
     	Combining  \eqref{PA} and \eqref{PB} gives
     	\begin{align*}
     		&A_\zeta^T(M^TP^*M)+(M^TP^*M)A_\zeta+Q_\zeta \\&-(M^TP^*M)B_\zeta{R}^{-1}B_\zeta^T(M^TP^*M)\\
     		=&M^T(A^T{P}^*+{P}^*A+{Q}-{P}^*B{R}^{-1}B^T{P}^*)M\\
     		=&0.
     	\end{align*}
     	Thus, $M^TP^*M$ is the solution to \eqref{areliz}.
     	
     	In conclusion, \eqref{areliz} admits a unique positive semi-definite solution ${P}_\zeta^*=M^TP^*M\geq 0$.
     \end{proof}

     From Theorem \ref{thm4}, we have $K_\zeta^*=-{R}^{-1}B_\zeta^T{P}_\zeta^*=-{R}^{-1}B^TP^*M=K^*M$.
     Now, consider the optimal controller $u^*_\zeta=K_\zeta^*\zeta$ for \eqref{dyz2}. Since  $M\zeta$ converges to $x$ exponentially,  $u^*_\zeta=K_\zeta^*\zeta=K^*M\zeta$ converges to the optimal controller  $u^*=K^*x$ for \eqref{lisys} exponentially. Thus, if we can solve the algebraic Riccati equation \eqref{areliz}, the LQR problem of \eqref{lisys} can then be solved with the output-based optimal controller $u^*_\zeta$.

     By now, we have converted the output-feedback LQR problem of \eqref{lisys} into the state-feedback LQR problem of the system \eqref{dyz2} with a known input matrix.   Nevertheless, since the existing state-based PI and VI methods  only apply to the case where  $(A,B)$ is stabilizable and $(A,\sqrt{{Q}})$ is  observable, they cannot be directly apply to  \eqref{areliz} where  the pair $(A_\zeta, B_\zeta)$ is stabilizable but the pair $(A_\zeta, \sqrt{{Q_y}}CM)$ is only detectable. For this reason, we will show in Appendix \ref{app1} and Appendix \ref{app2}, respectively,  that both the state-based PI and VI methods apply to the case where  $(A,B)$ is stabilizable and  $(A,\sqrt{{Q}})$ is detectable.

    Moreover, since $B_\zeta$ is known, we can further obtain an improved PI algorithm and an improved VI algorithm for the output-based data-driven LQR problem for \eqref{lisys} as follows.

     \subsection{An Improved PI Method}\label{sec3-1}
     Given \eqref{dyz2}, applying the Kleinman's algorithm \eqref{aleq1sin}  \cite{Kleinman}
     gives the following equations
     \begin{subequations} \label{aleq1z}
     	\begin{align}
     			0=&(A_\zeta^k)^T\bar{P}^P_{k} +\bar{P}^P_{k}A_\zeta^k+{Q}_\zeta+(\bar{K}^P_{k})^T{R}\bar{K}^P_{k}
     		\\ \label{aleq2z}
     		\bar{K}^P_{k+1}=&-{R}^{-1} B_\zeta^T\bar{P}^P_{k},
     \end{align}
     \end{subequations}
     	where $A_\zeta^k=A_\zeta+B_\zeta\bar{K}^P_{k}$ and
      $\bar{K}^P_{0}$ is such that $A_\zeta^0$ is Hurwitz.
     We have the following result.

     \begin{thm}\label{thm5}
     	Suppose  $(A_\zeta, B_\zeta)$ is stabilizable and $(A_\zeta, \sqrt{{Q_y}}CM)$ is detectable. Let ${P}_\zeta^*$ be the unique positive semi-definite solution of \eqref{areliz} and ${K}_\zeta^* = -{R}^{-1} B_\zeta^T {P}_\zeta^*$.
     Then  the Kleinman's algorithm \eqref{aleq1z} guarantees $\sigma(A_\zeta^{k}) \subset \mathbb{C}_-$, $\lim\limits_{k\to \infty} \bar{P}^P_{k}={P}_\zeta^*$ and $\lim\limits_{k\to \infty} \bar{K}^P_{k}={K}_\zeta^*$.

     \end{thm}
     \begin{proof}
     This theorem is a direct application of Theorem \ref{thmconpi} in Appendix \ref{app1}.
     \end{proof}

     We now derive PI algorithm for \eqref{dyz2}.  Rewrite \eqref{dyz2} as follows:
     \begin{align}\label{dyzpi}
     	\dot{\zeta}
     	=& A_\zeta\zeta+B_\zeta u\notag \\
     	=&A_\zeta^k \zeta-B_\zeta(\bar{K}^P_{k}\zeta-u) , t\geq t_0.
     \end{align}

     Since
     $B_\zeta$ is  a known matrix and
       $\zeta^TQ_\zeta\zeta=\zeta^TM^TC^TQ_yCM\zeta$ will exponentially converge to $y^TQ_yy$, from \eqref{dyzpi} and \eqref{aleq1z}, by replacing $\zeta^TQ_\zeta\zeta$ with $y^TQ_yy$, we obtain
     \begin{align}\label{piintz}
     &	|\zeta(t+\delta t)|_{\bar{P}_k^P}-|\zeta(t)|_{\bar{P}_k^P} \notag\\
     &	=\int_{t}^{t+\delta t}-|y|_{Q_y}-\zeta^T(\bar{K}_k^P)^TR\bar{K}_k^P\zeta\notag \\
     & -2\zeta^T(\bar{K}^P_k)^T B_\zeta^T \bar{P}_k^P  \zeta+ 2u^T B_\zeta^T \bar{P}_k^P \zeta d\tau .
     \end{align}

     Noting that
      there exists a constant matrix $N_{n_\zeta}\in \mathbb{R}^{{n_\zeta}^2\times \frac{{n_\zeta}({n_\zeta}+1)}{2}}$ such that $N_{n_\zeta}\text{vecs}(\bar{P}^P_k)=\text{vec}(\bar{P}^P_k)$,
    \eqref{piintz} implies
     \begin{align}\label{pilinearznew}
    	\tilde{\Psi}_k^P\text{vecs}(\bar{P}^P_k)=\tilde{\Phi}_k^P,
    \end{align}
       where $\tilde{\Psi}_k^P=[\delta_\zeta +2\Gamma_{\zeta\zeta}(I_{n_\zeta}\otimes ((\bar{K}^P_k)^T B_\zeta^T ))N_{n_\zeta} -2\Gamma_{\zeta u}(I_{n_\zeta}\otimes {B_\zeta^T})N_{n_\zeta}]$, $\tilde{\Phi}_k^P=-\Gamma_{yy}\text{vec}(Q_y)-\Gamma_{\zeta\zeta}\text{vec}((\bar{K}^P_{k})^T{R}\bar{K}^P_{k})$ and $n_\zeta=(m+p)n$.
    The solvability of \eqref{pilinearznew} is guaranteed by the following lemma.
    \begin{lem}
    	The matrix $\tilde{\Psi}_k^P$ has full column rank for {\color{black}$k\in\mathbb{N}$}, if
    	\begin{align}\label{rankconpiz}
    		\textup{rank}(\Gamma_{\zeta\zeta})=& \frac{n_\zeta(n_\zeta+1)}{2},
    	\end{align}
    	where $n_\zeta=(m+p)n$.
    \end{lem}
    \begin{proof}
    	The proof of this lemma is the same  as that of Lemma 6 of \cite{jiang2012computational} and is thus omitted.
    \end{proof}

     Thus, by iteratively solving \eqref{pilinearznew} and \eqref{aleq2z}, we can obtain the approximations   to ${P}_\zeta^*$  and $K_\zeta^*$.
     The improved PI algorithm is shown as Algorithm \ref{algpi}.

     \begin{rem}
     	{\color{black}   Compared with \eqref{pilinearout}, our approach leads to \eqref{pilinearznew} whose  number of the unknown variables is  reduced by $mn_\zeta$. As a result, the rank condition \eqref{rankconpiz} is also milder than \eqref{rankconpiout} since the column dimension of the matrix for the rank test has been reduced by $mn_\zeta$. Therefore, the improved PI  algorithm has significantly relaxed the rank condition and reduced the computational cost. This improvement is made possible by constructing the dynamics \eqref{dyz2} of $\zeta$.}
     \end{rem}

    \begin{algorithm}
    	\caption{The Improved output-based PI Algorithm}\label{algpi}
    	\begin{algorithmic}[1]
    		\State {\color{black}Find a feedback gain $\bar{K}^P_{0}$ such that $A_\zeta^0\triangleq A_\zeta+B_\zeta\bar{K}^P_{0}$ is Hurwitz.} Apply an initial input $u^0=\bar{K}^P_{0}\zeta+\delta$ with exploration noise $\delta$. Choose a proper $t_0>0$ such that $||e_\zeta||$ is small enough.
    		\State {\color{black}Collect} data from $t_0$ until the rank condition \eqref{rankconpiz} is satisfied. $k\leftarrow 0$.
    		\Repeat
    		\State Solve $\bar{P}^P_k$ from (\ref{pilinearznew}). $\bar{K}^P_{k+1}=-{R}^{-1} B_\zeta^T\bar{P}^P_{k}$. $k\leftarrow k+1$.
    		\Until{$||\bar{P}^P_k-\bar{P}^P_{k-1}||<\varepsilon$ with a sufficiently small constant $\varepsilon>0$}.
    		\State ${P}_\zeta^*\leftarrow \bar{P}^P_k$. $K_\zeta^*\leftarrow-{R}^{-1}B_\zeta^T{P}_\zeta^*$.
    		\State Obtain the following optimal controller
    		\begin{align}\label{controllerpi}
    			u^*=K_\zeta^*\zeta.
    		\end{align}
    	\end{algorithmic}
    \end{algorithm}

     For a simple comparison, consider a case with $n=5,m=5,p=1,n_\zeta=(m+p)n=30$. {\color{black}TABLE \ref{tab1}  compares the computational complexity and rank condition of
      the original output-based PI algorithm  presented in Section \ref{sec2-3}, and the improved output-based PI algorithm. }

\begin{table}[!ht]\color{black}
	\caption{Comparison of PI Algorithms}\label{tab1}
	\centering
	\begin{tabular}{|c|c|c|c|}
		\hline
		\multicolumn{1}{|c|}{\multirow{2}{*}{PI Algorithm}} &
		\multicolumn{1}{c|}{\multirow{2}{*}{Equation}} &
		Unknown &
		\multicolumn{1}{c|}{\multirow{2}{*}{Rank Condition}} \\
		& & Variables &  \\
		\hline
		Original Version & \eqref{pilinearout} & 615 & $\textup{rank}([\Gamma_{\zeta\zeta}, \Gamma_{\zeta u}])=615$ \\
		\hline
		Improved Version & \eqref{pilinearznew} & 465 & $\textup{rank}(\Gamma_{\zeta\zeta})=465$ \\
		\hline
	\end{tabular}
\end{table}

%
%
%
%
%

     \subsection{An Improved VI Method}\label{sec3-2}
	
In this subsection, we further improve the output-based  VI Algorithm \ref{vialgout} by reducing the computational cost and  relaxing the solvability condition in the case where $(A_\zeta, B_\zeta)$ is stabilizable and $(A_\zeta, \sqrt{{Q_y}}CM)$ is detectable. For this purpose, consider Algorithm \ref{vialgzmodel} where  $\bar{P}^V_k$ is an approximation to ${P}_\zeta^*$ in the $k$-th iteration, $\epsilon_k$ and $\varepsilon$ are as defined in Section \ref{sec2-2}, and $\{\bar{B}_q\}_{q=0}^{\infty}$ is a collection of bounded subsets in $\mathcal{P}^{n_\zeta}$ satisfying
	\begin{align}
		\bar{B}_q\subset \bar{B}_{q+1}, \; q\in \mathbb{N}, \; \lim\limits_{q\to \infty}\bar{B}_q=\mathcal{P}^{n_\zeta}.
	\end{align}

	\begin{algorithm}[H]
		\caption{{\color{black}Model-based} VI  Algorithm  for solving \eqref{areliz} }\label{vialgzmodel}
		\begin{algorithmic}[1]
			\State Choose $\bar{P}^V_0=(\bar{P}^V_0)^T\geq0$. $k,q\leftarrow 0$.
			\Loop
			\State $\tilde{\bar{P}}^V_{k+1}\leftarrow \bar{P}^V_k+\epsilon_k (A_\zeta^T\bar{P}^V_k+\bar{P}^V_kA_\zeta+{Q}_\zeta-\bar{P}^V_kB_\zeta{R}^{-1}B_\zeta^T\bar{P}^V_k)$.
			\If{$\tilde{\bar{P}}^V_{k+1}\notin \bar{B}_q$}
			\State $\bar{P}^V_{k+1}\leftarrow \bar{P}^V_0$. $q\leftarrow q+1$.
			\ElsIf{$||\tilde{\bar{P}}^V_{k+1}-\bar{P}^V_k||/\epsilon_k<\varepsilon$}
			\State \Return $\bar{P}^V_k$ as an approximation to ${P}_\zeta^*$.
			\Else
			\State $\bar{P}^V_{k+1}\leftarrow \tilde{\bar{P}}^V_{k+1}$.
			\EndIf
			\State $k \leftarrow k+1$.
			\EndLoop
		\end{algorithmic}
	\end{algorithm}
	
	We have the following result.
	
	 \begin{thm}\label{thm6}
	If $(A_\zeta, B_\zeta)$ is stabilizable and $(A_\zeta, \sqrt{{Q_y}}CM)$ is detectable, then  Algorithm \ref{vialgzmodel} is such that $\lim\limits_{k\to \infty}\bar{P}^V_k={P}_\zeta^*$.
	\end{thm}
	\begin{proof}
		 This theorem is a direct application of Theorem \ref{thmconvi} in Appendix \ref{app2}.
	\end{proof}
	
	\begin{rem}
	In contrast to 	VI Algorithm \ref{vialg1}	proposed in \cite{bian2016}, $\bar{P}^V_0$ in Algorithm \ref{vialgzmodel} can be any positive semi-definite matrix.
	\end{rem}
	
	 Now, let $\bar{H}^V_k=A_\zeta^T\bar{P}^V_k+\bar{P}^V_kA_\zeta+{Q}_\zeta$,  $ \bar{K}^V_k=-{R}^{-1}B_\zeta^T\bar{P}^V_k$.
	Then, from \eqref{dyz2}, using the relation $\zeta^TQ_\zeta\zeta=y^TQ_yy$, we obtain
	\begin{align}\label{intviz}
		|\zeta(t+\delta t)|_{\bar{P}^V_k}-&|\zeta(t)|_{\bar{P}^V_k} +\int_{t}^{t+\delta t} |y|_{Q_y} d\tau\notag \\=&\int_{t}^{t+\delta t}\lbrack |\zeta|_{\bar{H}^V_k} +2{u}^T B_\zeta^T\bar{P}^V_k \zeta \rbrack d\tau.
	\end{align}
	 Noting ${R} \bar{K}^V_{k}=- B_\zeta^T\bar{P}^V_{k}$ is known,
	\eqref{intviz}  implies
	\begin{align}\label{vilinearznew}
		\tilde{\Psi}^V\text{vecs}(\bar{H}^V_k)=\tilde{\Phi}^V_k,
	\end{align}
	where $\tilde{\Psi}^V=I_{\zeta\zeta}$ and $\tilde{\Phi}^V_k=\delta_{\zeta}\text{vecs}(\bar{P}^V_k)+I_{yy}\text{vecs}(Q_y)+2I_{\zeta u}	\text{vec}(\bar{K}^V_k)$.

	\begin{rem}\label{rem13}
		The definition of $\bar{K}^V_k$ in this section is different from that of Section \ref{sec2}. In Section \ref{sec2}, $\bar{K}^V_k$ is defined as $\bar{K}^V_k\triangleq -{R}^{-1}B^T{P}^V_kM$. In this section, $\bar{K}^V_k\triangleq -{R}^{-1}B_\zeta^T\bar{P}^V_k$. Since $B_\zeta$ and $\bar{P}^V_k$ are known matrices in each iteration, $\bar{K}^V_k\triangleq -{R}^{-1}B_\zeta^T\bar{P}^V_k$ in this section is a known matrix.
	\end{rem}

The solvability of \eqref{vilinearznew} is clearly guaranteed by the following condition.
\begin{lem}
	The matrix $\tilde{\Psi}^V$ has full column rank if
	\begin{align}\label{rankconviz}
		\textup{rank}(I_{\zeta\zeta})=& \frac{n_\zeta(n_\zeta+1)}{2},
	\end{align}
	where $n_\zeta=(m+p)n$.
\end{lem}

{\color{black}By iteratively solving \eqref{vilinearznew}, we can obtain $\bar{H}^V_k$. Thus, the updating law of $\tilde{\bar{P}}^V_{k+1}$ in Algorithm \ref{vialgzmodel} can be replaced by $\tilde{\bar{P}}^V_{k+1}\leftarrow \bar{P}^V_k+\epsilon_k (\bar{H}^V_k-\bar{P}^V_kB_\zeta{R}^{-1}B_\zeta^T\bar{P}^V_k)$,  which relies on known $\bar{H}^V_k$ and known matrix $B_{\zeta}$.
The above improved VI algorithm is summarized as Algorithm \ref{vialgz}.}

\begin{rem}
	Similarly to the PI algorithm, by obtaining \eqref{vilinearznew}, we have reduced the number of the unknown variables of \eqref{vilinearout} by $mn_\zeta$. As a result, the rank condition \eqref{rankconviz} for our improved VI algorithm is milder than the original rank condition \eqref{rankconviout} since the column dimension of the matrix for the rank test has been reduced by $mn_\zeta$. Thus, our improvement not only significantly saves the computational cost, but also remarkably relaxes the rank requirement.
\end{rem}

\begin{algorithm}
	\caption{ The improved output-based VI Algorithm }\label{vialgz}
	\begin{algorithmic}[1]
		\State  Apply any locally essentially bounded initial input $u^0$. Choose a proper $t_0>0$ such that $||e_\zeta||$ is small enough.
		{\color{black}Collect} data  from $t_0$ until the rank condition \eqref{rankconviz} is satisfied.
		\State Choose $\bar{P}^V_0=(\bar{P}^V_0)^T\geq0$. $k,q\leftarrow 0$.
		\Loop
		\State Solve $\bar{H}^V_k$ from \eqref{vilinearznew}.
		\State $\tilde{\bar{P}}^V_{k+1}\leftarrow \bar{P}^V_k+\epsilon_k (\bar{H}^V_k-\bar{P}^V_kB_\zeta{R}^{-1}B_\zeta^T\bar{P}^V_k)$.
		\If{$\tilde{\bar{P}}^V_{k+1}\notin\bar{B}_q$}
		\State $\bar{P}^V_{k+1}\leftarrow \bar{P}^V_0$. $q\leftarrow q+1$.
		\ElsIf{$||\tilde{\bar{P}}^V_{k+1}-\bar{P}^V_k||/\epsilon_k<\varepsilon$}
		\State \Return $\bar{P}^V_k$ and $-{R}^{-1}B_\zeta^T\bar{P}^V_k$ as  approximations to ${P}_\zeta^*$  and $K_\zeta^*$, respectively.
		\Else
		\State $\bar{P}^V_{k+1}\leftarrow \tilde{\bar{P}}^V_{k+1}$.
		\EndIf
		\State $k \leftarrow k+1$.
		\EndLoop
		\State Obtain the following optimal controller
		\begin{align}\label{controllervi}
			u^*=K_\zeta^*\zeta.
		\end{align}
	\end{algorithmic}
\end{algorithm}

To see the effect of our improvement, consider the simple case with $n=5,m=5,p=1,n_\zeta=(m+p)n=30$ again. {\color{black}TABLE \ref{tab11} compares the computational complexity and rank condition between the original output-based VI algorithm presented in Section \ref{sec2-3}, and the improved output-based VI algorithm.} It can be seen that the unknown variables and the column dimension of matrix for rank test have both decreased by around $24.4\% $.

\begin{table}[!ht]\color{black}
	\caption{Comparison of VI Algorithms}\label{tab11}
	\centering
	\begin{tabular}{|c|c|c|c|}
		\hline
		\multicolumn{1}{|c|}{\multirow{2}{*}{VI Algorithm}} &
		\multicolumn{1}{c|}{\multirow{2}{*}{Equation}} &
		Unknown &
		\multicolumn{1}{c|}{\multirow{2}{*}{Rank Condition}} \\
		& & Variables &  \\
		\hline
		Original Version & \eqref{vilinearout} & 615 & $
		\textup{rank}([I_{\zeta\zeta}, I_{\zeta u}])=615$ \\
		\hline
		Improved Version & \eqref{vilinearznew} & 465 &  $\textup{rank}(I_{\zeta\zeta})=465$ \\
		\hline
	\end{tabular}
\end{table}

   \section{Application}\label{sec5}
   In this section, we use two examples to illustrate the effectiveness of our proposed output-based data-driven algorithms.

   \subsection{Example 1}
   Consider the example given in \cite{rizvi2019} which comes from the load frequency {\color{black}control} of power systems \cite{vrabie2009adaptive} as follows:
   \begin{align}\label{ex1}
   	\begin{split}
   		A&=\begin{bmatrix}
   			-0.0665 &8 &0 &0\\
   			0 &-3.663 &3.663 &0\\
   			-6.86& 0 &-13.736& -13.736\\
   			0.6 &0 &0& 0
   		\end{bmatrix},
   	\\ B&=\begin{bmatrix}
   	0\\0\\13.736\\0
   	\end{bmatrix},C=\begin{bmatrix}
   	1&0&0&0
   	\end{bmatrix}.
   \end{split}
   \end{align}
   It can be  checked that Assumption \ref{assctrl}-\ref{ass2} are all satisfied for this example. Moreover, the eigenvalues of $A$ are $-14.8527, -1.4798, -0.5665 \pm 3.2661i$. Thus, $A$  is Hurwitz and $K={\bf 0}$ is such that $A+BK$ is stable.  In this way, an initial stabilizing gain is obtained as $\bar{K}^P_0=KM={\bf 0}$, and the improved output-based PI Algorithm \ref{algpi} can be applied to this example.

   Let $Q_y=1,R=1$, $\Lambda(s)=(s+5)(s+6)(s+7)(s+8)$. It can be calculated that
   \begin{align*}
   	P^*=&\begin{bmatrix}
   		0.3135   & 0.2864  &  0.0509 &   0.1912\\
   		0.2864    &0.4156  &  0.0903 &   0.0789\\
   		0.0509  &  0.0903  &  0.0210 &   0\\
   		0.1912  &  0.0789  &  0  &  1.1868
   	\end{bmatrix}\\
   	K^*=&\begin{bmatrix}
   		-0.6994 &  -1.2404 &  -0.2890  & 0
   	\end{bmatrix}\\
   	L=&\begin{bmatrix}
   		8.5345 & 6.3795   & -9.1734 &-3.5737
   	\end{bmatrix}^T\\
   	M=&[M_u , M_y],
   	 \end{align*}
   	 where
    \begin{align*}	M_u=&10^{3}\times\begin{bmatrix}
   		0   & 0.4025   &      0    &     0\\
   		0   & 0.4328  &  0.0503      &   0\\
   		0  &  1.1338  &  0.1685   & 0.0137\\
   		1.6800    &     0      &   0      &   0\\
   	\end{bmatrix}\\
   	M_y=&10^{3}\times\begin{bmatrix}
   		1.4385 &  0.8616 &  0.1995&   0.0085\\
   		-0.2457  &  -0.0311  & 0.0545 &  0.0064\\
   		-0.6663  &  -0.4541 &   -0.0437  &  -0.0092\\
   		-0.2134  &  -0.0642  &  -0.0573   & -0.0036
   	\end{bmatrix}.
   	   \end{align*}

   Choose the initial conditions as $x(0)=[1,1,1,1]^T,\zeta(0)={\bf 0}$. Applying an exploring initial input $u^0=\bar{K}^P_0\zeta+\delta$ with $\bar{K}^P_0={\bf 0}$ and $\delta=20[\sin(t)+\sin(7t)+\sin(10t)+\sin(16t)]$ to the system at $t_0=3~secs$ for data collecting. Let {\color{black}$t_q=3+0.1q, q=0,1,\cdots,s$} with $s=45$. Then the data collected in $[3,7.5]secs$ is  such that \eqref{rankconpiz} is satisfied. Start the policy iteration process of Algorithm \ref{algpi} with $\bar{K}^P_0$. Then it takes 8 steps for the convergence criterion $||\bar{P}^P_k-\bar{P}^P_{k-1}||<\varepsilon$ with $\varepsilon=0.01$ to be satisfied. {\color{black}Fig. \ref{iterafigpi1} compares $\bar{P}^P_k,\bar{K}^P_k$ with their target values $P_\zeta^*,K_\zeta^*$, respectively.  The y-axis in Fig. \ref{iterafigpi1} stands for normalized errors $\frac{||\bar{P}^P_k-P_\zeta^*||}{||P_\zeta^*||}, \frac{||\bar{K}^P_k-K_\zeta^*||}{||K_\zeta^*||}$ with unit being ``1''}.   By calculation, we obtain {\color{black}the normalized error for estimating $K^*_\zeta$ as follows:}
   \begin{align*}
   	\frac{||\bar{K}^P_8-K_\zeta^*||}{||K_\zeta^*||}=0.0002   .
   \end{align*}
   Thus, our proposed improved output-based PI Algorithm \ref{algpi} works satisfactorily.
   {\color{black}Furthermore, we repeat the simulation 1000 times on MATLAB running on a MacBook Pro with the processor being Apple M2 Pro. The average computing time needed for convergence is  $0.0106$ s.}

   It is noted that the rank condition \eqref{rankconpiout} is not satisfied for the data collected from  $[3,7.5]secs$. Thus, the output-based PI algorithm of \cite{rizvi2019} does not apply to  this scenario.

   \begin{figure}
   	\centering
   	\includegraphics[width=8.5cm]{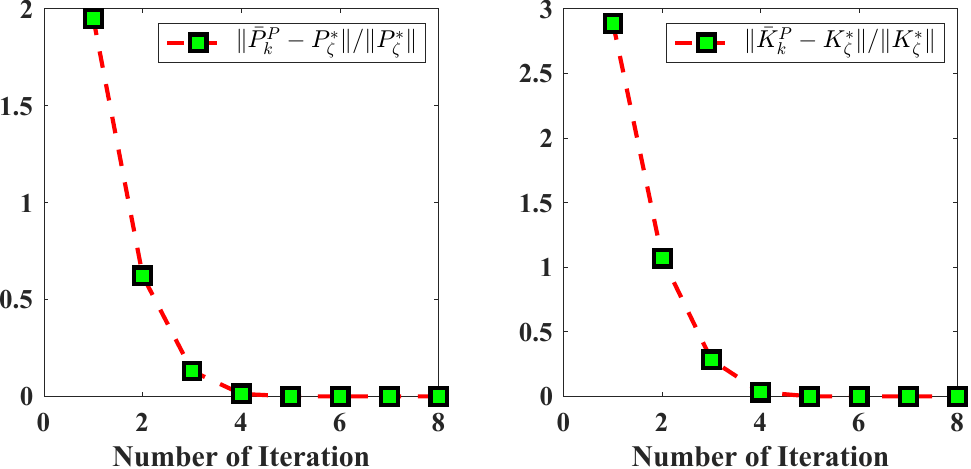}
   	\caption{Comparison of $\bar{P}^P_{k},\bar{K}^P_{k}$ and their target values $P_\zeta^*,K_\zeta^*$ for Algorithm \ref{algpi}. {\color{black}The normalized errors $\frac{||\bar{P}^P_k-P_\zeta^*||}{||P_\zeta^*||}, \frac{||\bar{K}^P_k-K_\zeta^*||}{||K_\zeta^*||}$ reach the desired accuracy within 8 steps}.}
   	\label{iterafigpi1}
   \end{figure}

   \subsection{Example 2}
Consider a linear system in the form of \eqref{lisys} with
\begin{align}\label{ex2}
	\begin{split}
		A=&\begin{bmatrix}
			-11 & 30\\-4& 11
		\end{bmatrix},B=\begin{bmatrix}
		10\\4
		\end{bmatrix},\\
		C=&\begin{bmatrix}
			1 &0
		\end{bmatrix}.
	\end{split}
\end{align}
It can be verified  that Assumptions \ref{ass1} and \ref{ass2} are both satisfied for this system, but Assumption \ref{assctrl} is not satisfied. Thus, the output-based algorithms in \cite{rizvi2019} does not apply to this example. Nevertheless, our improved output-based algorithms still apply to this example.

The eigenvalues of $A$ are $\pm 1$. Thus, this is an unstable system, and we will apply the improved output-based VI Algorithm \ref{vialgz} to this example.

Let $Q_y=1,R=1$, $\Lambda(s)=(s+6)(s+7)$. Simple calculations give
\begin{align*}
	P^*=&\begin{bmatrix}
		0.5905   &-1.5000\\
		-1.5000  &  4.5000
	\end{bmatrix}\\
	K^*=&\begin{bmatrix}
		0.0950 &  -3.0000
	\end{bmatrix}\\
	L=&\begin{bmatrix}
		13.0000&
		6.2000
	\end{bmatrix}^T\\
	M_u=&\begin{bmatrix}
		10.0000 &  10.0000\\
		-6.0000  &  4.0000
	\end{bmatrix}\\
	M_y=&\begin{bmatrix}
		43.0000 & 13.0000\\
		16.2000 &  6.2000
	\end{bmatrix}\\
	M=&[M_u , M_y].
\end{align*}

The initial conditions are set as $x(0)=[1,1]^T,\zeta(0)={\bf 0}$.  An exploring initial input $u^0={K}_\zeta^0\zeta+\delta$ with ${K}_\zeta^0=[34.0000 ,  -6.0000 , -21.8000  ,-11.8000]$
and $\delta=20\sin(3t)$ is injected to the system at $t_0=4~secs$ for data collecting. Let {\color{black}$t_q=4+0.05q, q=0,1,\cdots,s$} with $s=15$. Then  the data collected in $[4,4.75]secs$ is such that  \eqref{rankconviz} is satisfied.
Let $\bar{P}^V_0$ be a positive semi-definite matrix as follows:
\begin{align*}
	\bar{P}^V_0=\begin{bmatrix}
		1&0&0&0\\0&1&0&0\\0&0&1&0\\0&0&0&0
	\end{bmatrix}\geq 0.
\end{align*}
Define  $\bar{B}_q$  as $\bar{B}_q=\{H\geq0 \:|\;|H|<1000(q+1)\}$ for $H=H^T\in \mathbb{R}^{4\times4}, q\in \mathbb{N}$ and $\epsilon_k=\frac{5}{k}$. Then, under the convergence criterion $||\tilde{\bar{P}}^V_{k+1}-\bar{P}^V_k||/\epsilon_k<\varepsilon$ with $\varepsilon=0.01$, it takes 1860 steps of iteration for Algorithm \ref{vialgz} to converge.  {\color{black} Fig. \ref{iterafigvi1} compares  $\bar{P}^V_k,\bar{K}^V_k$ with their target values $P_\zeta^*,K_\zeta^*$, respectively.  The y-axis in Fig. \ref{iterafigvi1}  stands for normalized errors $\frac{||\bar{P}^V_k-P_\zeta^*||}{||P_\zeta^*||}$ and $ \frac{||\bar{K}^V_k-K_\zeta^*||}{||K_\zeta^*||}$ with unit being ``1''}.

By calculation, we obtain {\color{black}the normalized error for estimating $K^*_\zeta$ as follows:}
\begin{align*}
	\frac{||\bar{K}^V_{1860}-K_\zeta^*||}{||K_\zeta^*||}=1.1266\times 10^{-4}.
\end{align*}
Thus, our proposed improved output-based VI Algorithm \ref{vialgz} works satisfactorily.
{\color{black}We also repeat the simulation 1000 times on MATLAB running on a MacBook Pro with the processor being Apple M2 Pro. The average computing time needed for convergence of Algorithm \ref{vialgz} is  $0.0345$ s.}

It is also noticed that the rank condition \eqref{rankconviout} is not satisfied for the data collected from  $[4,4.75]secs$. Thus, the output-based VI algorithm of \cite{rizvi2019} does not  apply to  this scenario.

\begin{figure}
	\centering
	\includegraphics[width=8.5cm]{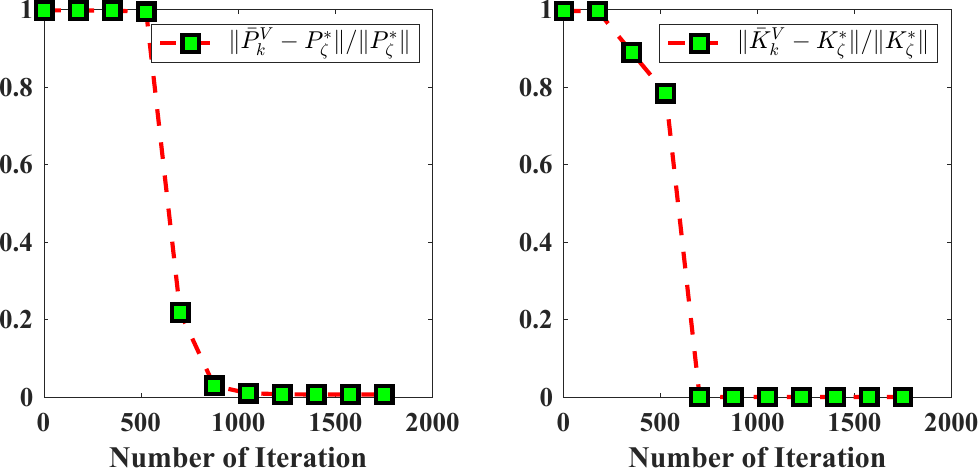}
	\caption{Comparision of $\bar{P}^V_{k}, \bar{K}^V_{k}$ and their target values $P_\zeta^*,K_\zeta^*$ for Algorithm \ref{vialgz}. {\color{black}The normalized errors $\frac{||\bar{P}^V_k-P_\zeta^*||}{||P_\zeta^*||}$ and $ \frac{||\bar{K}^V_k-K_\zeta^*||}{||K_\zeta^*||}$ reach the desired accuracy in 1860 steps.}}
	\label{iterafigvi1}
\end{figure}

\begin{table}[htbp]\color{black}
	\caption{Simulation Results of Our Improved Output-Based ADP Algorithms}
	\centering
	\begin{tabular}{c | c | c}
		\hline\hline
		~ & \emph{Improved PI Algorithm} & \emph{Improved VI Algorithm } \\
		\hline
		 Applied Scenario&Example 1 & Example 2
		\\ \hline
		\makecell[c]{Iteration Steps\\for Convergence}& 8& 1860
		\\ \hline
		 \makecell[c]{Normalized Error \\for Estimating $K^*_\zeta$}& $2\times 10^{-4}$&$1.1266\times 10^{-4}$
		\\ \hline
		\makecell[c]{Experiment \\Repeated Times}&1000 & 1000\\ \hline
		\makecell[c]{Average Computing \\ Time for Iterations}&0.0106 s& 0.0345 s\\
		\hline\hline
	\end{tabular}
	\label{tab:comtime}
\end{table}

	\section{Conclusions}\label{sec6}
	
	In this paper, we have further investigated the problem of designing the output-based ADP PI and VI algorithms for continuous-time linear systems based on the results of \cite{rizvi2019} and \cite{rizvi2023}. Compared with the existing results, this paper provides the following four new features.  First, we have removed the full row rank requirement on the parameterization matrix $M$, thus removing  the main hurdle for
	applying this method.
	Second,
	we have derived  a new sequence of linear equations for implementing both output-based PI and VI algorithms whose
	number of unknown variables is significantly less than that of the existing output-based PI and VI algorithms, thus  improving the computational efficiency of the existing output-based PI and VI algorithms.
	Moreover, the solvability condition of this sequence of equations is also less stringent than the one in the literature.
	Third,  since the existing state-based PI or VI algorithm only applies to the case where the Riccati equation admits a unique positive definite solution, we have shown that both the state-based PI algorithm   and the VI algorithm in the literature can be generalized to the case where the Riccati equation only admits a unique positive semi-definite solution.
	Fourth, as a result of the third feature,  we have broadened the scope of applicability of the existing techniques from controllable linear systems to stabilizable linear systems.  The future work will focus on developing output-based data-driven algorithms for linear multi-agent systems.

\appendices
\section{Extended Model-Based PI Approach}\label{app1}

In this appendix, we will extend the Kleinman's algorithm \eqref{aleq1sin} to the case where $(A,B)$ is stabilizable and $(A,\sqrt{{Q}})$ is detectable.

To begin with, we summarize the following theorem from \cite{kucera1972} regarding the solution to the algebraic Riccati equation \eqref{areli}.

\begin{thm}\label{thmpsd}
	Suppose $(A, B)$ is stabilizable and  $(A,\sqrt{{Q}})$ is detectable. Then,  there exists a unique positive semi-definite solution $P^*\geq 0$ to \eqref{areli}.
\end{thm}


Now, for any  $ K \in  \mathbb{R}^{m \times n}$,    define a cost matrix associated  with $K$ as follows:
{\small\begin{align}\label{vkt}
	V_K(t,T)\triangleq\int_{t}^{T}e^{(A+BK)^T(\tau-t)}(Q+K^TRK)e^{(A+BK)(\tau-t)}d\tau.
\end{align}}

Then, we have the following result.
\begin{lem} \label{lemerrV} For any  $ K_1, K_2  \in  \mathbb{R}^{m \times n}$,
	the following relations hold:
	\begin{align}\label{v1-v21}
		\begin{split}
			V_1-V_2
			=&\int_{0}^{\infty} e^{(A+BK_2)^T\tau}[(K_1-K_2)^TR(K_1-K_2)\\
			&+(K_1-K_2)^T(B^TV_1+RK_2)\\
			&+(B^TV_1+RK_2)^T(K_1-K_2)] e^{(A+BK_2)\tau}d\tau
		\end{split}
	\end{align}
	
	\begin{align}\label{v1-v22}
		\begin{split}
			V_1-V_2
			=&\int_{0}^{\infty} e^{(A+BK_1)^T\tau}[(K_1-K_2)^TR(K_1-K_2)\\
			&+(K_1-K_2)^T(B^TV_2+RK_2)\\
			&+(B^TV_2+RK_2)^T(K_1-K_2)] e^{(A+BK_1)\tau}d\tau,
		\end{split}
	\end{align}
	where $V_1=V_{K_1}(0,\infty), V_2=V_{K_2}(0,\infty)$.
\end{lem}
\begin{proof}
	The proof of this lemma is motivated by Appendix C of \cite{kleinmanphd}.
	First, from \eqref{vkt}, we have
	\begin{align}
		\begin{split}
			\frac{dV_K(t,T)}{dt}=&-(Q+K^TRK)-(A+BK)^TV_K(t,T)\\
			&-V_K(t,T)(A+BK).
		\end{split}
	\end{align}
	Thus,
	\begin{align}\label{ddeltav}
		\begin{split}
			&\frac{d(V_{K_1}(t,T)-V_{K_2}(t,T))}{dt}\\
			=&(A+BK_2)^TV_{K_2}(t,T)+V_{K_2}(t,T)(A+BK_2)\\&+K_2^TRK_2
			-(A+BK_1)^TV_{K_1}(t,T)\\&-V_{K_1}(t,T)(A+BK_1)-K_1^TRK_1\\
			=&(A+BK_2)^T(V_{K_2}(t,T)-V_{K_1}(t,T))\\&+(V_{K_2}(t,T)-V_{K_1}(t,T))(A+BK_2)\\
			&-(K_1-K_2)^TB^TV_{K_1}(t,T)-V_{K_1}(t,T)B(K_1-K_2)\\
			&-(K_1-K_2)^TR(K_1-K_2)-(K_1-K_2)^TRK_2\\&-K_2^TR(K_1-K_2)\\
			=&-(A+BK_2)^T(V_{K_1}(t,T)-V_{K_2}(t,T))\\&-(V_{K_1}(t,T)-V_{K_2}(t,T))(A+BK_2)\\
			&-(K_1-K_2)^TR(K_1-K_2)\\&-(K_1-K_2)^T(B^TV_{K_1}(t,T)+RK_2)\\&-(B^TV_{K_1}(t,T)+RK_2)^T(K_1-K_2).
		\end{split}
	\end{align}
	
	Let $\delta V_K(t,T)\triangleq V_{K_1}(t,T)-V_{K_2}(t,T)$.  Then, with the boundary condition $\delta V_K(T,T)=0$, the solution to \eqref{ddeltav} is given by
	\begin{align}\label{deltaVK}
		\begin{split}
			\delta V_K(t,T)=&\int_{t}^{T}e^{(A+BK_2)^T(\tau-t)}[(K_1-K_2)^TR(K_1-K_2)\\&+(K_1-K_2)^T(B^TV_{K_1}(\tau,T)+RK_2)\\&+(B^TV_{K_1}(\tau,T)+RK_2)^T(K_1-K_2)]
			\\&e^{(A+BK_2)(\tau-t)}d\tau.
		\end{split}
	\end{align}
	
	From  \eqref{vkt},  we have  $V_K(t,\infty)=V_K(0,\infty)$ for any $t$ and $K$. Thus,
	in \eqref{deltaVK}, let $t=0,T=\infty$, we directly obtain \eqref{v1-v21}. Similarly, we can also obtain \eqref{v1-v22}.
\end{proof}
%
%

\begin{lem}\label{lemonlyif}
	Suppose $(A, B)$ is stabilizable and  $(A,\sqrt{{Q}})$ is detectable. Then, $V_K(0,\infty)$ is finite if and only if $A+BK$ is Hurwitz.
\end{lem}
\begin{proof}
	The proof of this lemma is inspired by that of Lemma 12.1 in \cite{Glad2010book}.
	
	{\color{black}First, if $A+BK$ is Hurwitz, then, there exist two positive constants $\alpha,\beta>0$ such that $||e^{(A+BK)t}||\leq \beta e^{-\alpha t}, \forall t\geq 0$.
	Since
	\begin{align*}
		V_K(0,\infty)&=\int_{0}^{\infty}e^{(A+BK)^Tt}(Q+K^TRK)e^{(A+BK)t}dt,
	\end{align*}
	which implies
	\begin{align*}
		||V_K(0,\infty)||
		&\leq \int_{0}^{\infty}||Q+K^TRK||\beta^2e^{-2\alpha t}dt\\
		&<\infty,
	\end{align*}
	$V_K(0,\infty)$ is finite if $A+BK$ is Hurwitz.}
	
	We prove the only if part by contradiction.
	Suppose $X=V_K(0,\infty)$ is finite, and  $A+BK$ has an eigenvalue $\lambda$ with non-negative real part. Let  $v \neq 0$ be such that  $(A+BK)v=\lambda v$ and
	$\bar{v}$ be the complex conjugate of
	$v$. Then,  $e^{(A+BK)t}v=e^{\lambda t}v$, which implies
	\begin{align*}
		\bar{v}^TXv=\bar{v}^T(Q+K^TRK)v\int_{0}^{\infty}e^{2Re{(\lambda)} t}dt,
	\end{align*}
	where $Re(\lambda)\geq 0$ denotes the real part of $\lambda$.
	Thus $\bar{v}^T(Q+K^TRK)v=0$ since  the real part of $\lambda$ is non-negative.
	
	Notice $Q+K^TRK=\begin{bmatrix}
		\sqrt{Q}^T &K^T\sqrt{R}^T
	\end{bmatrix}\begin{bmatrix}
		\sqrt{Q}\\ \sqrt{R}K
	\end{bmatrix}\geq 0$, where $\sqrt{R}$ is an invertible matrix since $R>0$.
	Let $L$ be such that $A-L\sqrt{Q}$ is Hurwitz. Then,
	\begin{align*}
		&(A+BK)-\begin{bmatrix}
			L& B\sqrt{R}^{-1}
		\end{bmatrix}\begin{bmatrix}
			\sqrt{Q}\\ \sqrt{R}K
		\end{bmatrix}\\=&A-L\sqrt{Q},
	\end{align*}
	which means that the pair $(A+BK, \begin{bmatrix}
		\sqrt{Q}\\ \sqrt{R}K
	\end{bmatrix})$ is detectable.
	
	{\color{black}Nevertheless, since $\bar{v}^T(Q+K^TRK)v=0$ implies $\begin{bmatrix}
		\sqrt{Q}\\ \sqrt{R}K
	\end{bmatrix}v=0$, which together with the fact that $(A+BK)v=\lambda v$ implies $\begin{bmatrix}
	\lambda I-(A+BK)\\ \begin{bmatrix}
		\sqrt{Q}\\ \sqrt{R}K
	\end{bmatrix}
	\end{bmatrix}v={0}, v\neq { 0}$. Thus, rank$(\begin{bmatrix}
	\lambda I-(A+BK)\\ \begin{bmatrix}
		\sqrt{Q}\\ \sqrt{R}K
	\end{bmatrix}
	\end{bmatrix})<n$.
	Since $Re(\lambda)\geq0$, by PBH test,   the pair $(A+BK, \begin{bmatrix}
		\sqrt{Q}\\ \sqrt{R}K
	\end{bmatrix})$ is not detectable, which
	contradicts  the detectability of $(A+BK, \begin{bmatrix}
		\sqrt{Q}\\ \sqrt{R}K
	\end{bmatrix})$. }Thus, $A+BK$ must be Hurwitz when $X=V_K(0,\infty)$ is finite.
	
	Besides, in this situation, $X$ is the unique positive semidefinite solution to the following Lyapunov equation:
	\begin{align*}
		(A+BK)^TX+X(A+BK)+Q+K^TRK=0.
	\end{align*}
	
\end{proof}

With Lemma \ref{lemerrV} and Lemma \ref{lemonlyif}, we are now ready to {\color{black}prove} that the Kleinman's algorithm \eqref{aleq1sin} can be extended to the case where $(A,B)$ is stabilizable and $(A,\sqrt{{Q}})$ is detectable as follows.

\begin{thm}\label{thmconpi}
	Suppose $(A, B)$ is stabilizable and  $(A,\sqrt{{Q}})$ is detectable. Then, starting from  any $K_0^P$ such that $A+BK_0^P$ is Hurwitz,  the Kleinman's algorithm \eqref{aleq1sin} guarantees the following properties:
	\begin{enumerate}
		\item $\sigma(A_{k}) \subset \mathbb{C}_-$;
		\item ${P}^*\leq {P}^P_{k+1}\leq {P}^P_k$;
		\item $\lim\limits_{k\to \infty}{K}^P_k={K}^*,\lim\limits_{k\to \infty}{P}^P_k={P}^*$.
	\end{enumerate}
\end{thm}
\begin{proof}
	The proof of Theorem \ref{thmconpi} is similar to the method in \cite{Kleinman}. First, for convenience, let $V_k=V_{K_k^P}(0,\infty)$. Then, by Lemma \ref{lemonlyif},
	$V_k$ is finite if and only if $A_k\triangleq A+BK_k^P$ is Hurwitz. In this case, $V_k$ is the unique positive semi-definite solution to the following Lyapunov equation
	\begin{align}
		A_k^TV_k+V_kA_k+Q+(K_k^P)^TRK_k^P=0.
	\end{align}
	That is, $V_k={P}^P_k$ when $A_k$ is Hurwitz.
	
	Now, let  $V_0={P}^P_0$. Using \eqref{aleq1sin} and \eqref{v1-v21} gives
	\begin{align}
		\begin{split}
			V_0-V_1=&\int_{0}^{\infty} e^{(A+BK_1^P)^T\tau}[(K_0^P-K_1^P)^TR(K_0^P-K_1^P)]\\& e^{(A+BK_1^P)\tau}d\tau\geq 0,
		\end{split}
	\end{align}
	which implies $V_0\geq V_1 \geq 0$. Hence, $V_1$ is finite and $K_1^P$ updated from \eqref{aleq1sin} is stabilizing. By induction, we conclude that $\sigma(A_{k}) \subset \mathbb{C}_-, V_k={P}^P_k$ and $ {P}^P_{k}\geq {P}^P_{k+1}\geq 0$ where $ k=0,1,\cdots$. Since ${P}^P_k$ is monotonically decreasing with a lower bound, the series $\{{P}^P_k,k=0,1,\cdots\}$ must have a limit. Taking the limit on both sides of \eqref{aleq1sin} yields $\lim\limits_{k\to \infty}{K}^P_k={K}^*$ and $\lim\limits_{k\to \infty}{P}^P_k={P}^*$.
	
	On the other hand, using \eqref{v1-v22}, we obtain
	\begin{align}
		\begin{split}
			V_k-P^*=&\int_{0}^{\infty} e^{(A+BK_k^P)^T\tau}[(K_k^P-K^*)^TR(K_k^P-K^*)]\\& e^{(A+BK_k^P)\tau}d\tau\geq 0,
		\end{split}
	\end{align}
	which implies $V_k={P}^P_k\geq P^*$. The proof is thus completed.
	
\end{proof}

\section{Extended Model-Based VI Approach}\label{app2}

In this appendix, we further {\color{black}prove} that the VI Algorithm \ref{vialg1} can also be extended to the case where  $(A,B)$ is stabilizable and $(A,\sqrt{{Q}})$ is detectable.
Moreover, we will prove that the initial value of ${P}^V_0$ in Algorithm \ref{vialg1} can be any positive semidefinite matrix as shown in
the following Algorithm \ref{vialg1ex}.
\begin{algorithm}
	\caption{{\color{black}Model-based} VI  Algorithm for solving \eqref{areli} }\label{vialg1ex}
	\begin{algorithmic}[1]
		\State Choose ${P}^V_0=({P}^V_0)^T\geq0$. $k,q\leftarrow 0$.
		\Loop
		\State $\tilde{P}^V_{k+1}\leftarrow {P}^V_k+\epsilon_k (A^T{P}^V_k+{P}^V_kA-{P}^V_kB{R}^{-1}B^T{P}^V_k+{Q})$.
		\If{$\tilde{P}^V_{k+1}\notin B_q$}
		\State ${P}^V_{k+1}\leftarrow {P}^V_0$. $q\leftarrow q+1$.
		\ElsIf{$||\tilde{P}^V_{k+1}-{P}^V_k||/\epsilon_k<\varepsilon$}
		\State \Return ${P}^V_k$ as an approximation to ${P}^*$.
		\Else
		\State ${P}^V_{k+1}\leftarrow \tilde{P}^V_{k+1}$.
		\EndIf
		\State $k \leftarrow k+1$.
		\EndLoop
	\end{algorithmic}
\end{algorithm}

We first summarize  Theorem 17 of \cite{kucera1973} as follows.
\begin{lem}\label{lemdifric}
	Suppose  $(A,B)$ is stabilizable and $(A,\sqrt{{Q}})$ is detectable. Let $P^*$ be the unique positive semi-definite solution to \eqref{areli}.  Then,  for any $P(0)\geq 0$,
	the solution $P (t)$ of the following differential Riccati equation
	\begin{align}\label{dreq}
		\begin{split}
			\dot{P}=A^TP+PA+Q-PBR^{-1}B^TP
		\end{split}
	\end{align}
	satisfies $\lim\limits_{t\to \infty}P(t)=P^*$.
\end{lem}

{\color{black}\begin{rem}
		Since $P^*\geq 0$ is on the boundary of $\mathcal{P}^n$,  Lemma \ref{lemdifric} does not guarantee that $P^*$ is a locally asymptotically stable equilibrium  of \eqref{dreq}.
		Proposition 3.6 in \cite{bian2019} further shows that $P^* \geq 0$ is a  locally asymptotically stable equilibrium point for \eqref{dreq}.
			To establish our main result, we need to further prove that  any $P(0) \geq 0$ is in the interior of the domain of attraction of $P^*$ when $(A,B)$ is stabilizable and $(A,\sqrt{{Q}})$ is detectable.  For this purpose, an extended version of Lemma \ref{lemdifric} given by \cite{kailath1976} is summarized as follows.
\end{rem}}

\begin{lem}\label{lemdifricex}
	Suppose $(A,B)$ is stabilizable and $(A,\sqrt{{Q}})$ is detectable. Let
	$P^a$ be the unique positive semi-definite solution to the following algebraic Riccati equation:
	\begin{align*}
		{P}^aA^T+A{P}^a+B{R}^{-1}B^T-{P}^a{Q}{P}^a=0
	\end{align*}
	and $P^*$ be 	the unique positive semi-definite solution to \eqref{areli}. Then,  for any $P(0)\in \mathcal{S}^n$ such that
	\begin{align*}
		x^T[{P}^aP(0){P}^a+{P}^a]x>0 ~~~ \textup{for all } x \textup{ such that } P^ax\neq 0,
	\end{align*}
	the solution of the differential Riccati equation \eqref{dreq}
	satisfies $\lim\limits_{t\to \infty}P(t)=P^*$.
\end{lem}

To facilitate our further discussion, a useful inequality is given in the following lemma.
\begin{lem}\label{lemnormineq}
	For any  $U=\begin{bmatrix}
		U_{11}&U_{12}\\U_{21}& U_{22}
	\end{bmatrix}\in \mathcal{S}^n$ with $U_{11}\in \mathcal{S}^m, 1\leq m\leq n$ being a top principle minor of $U$ and $U_{12},U_{21},U_{22}$ being block matrices with compatible size, the following inequality holds:
	\begin{align}\label{normineq}
		||U_{11}||\leq ||U||.
	\end{align}
\end{lem}
\begin{proof}
	For any $v\in \mathbb{R}^m$, $v^TU_{11}v=\begin{bmatrix}
		v\\{\bf 0}
	\end{bmatrix}^TU\begin{bmatrix}
		v\\{\bf 0}
	\end{bmatrix}$. Thus, $\max_{||v||=1}|v^TU_{11}v|\leq \max_{||w||=1}|w^TUw|$. As a result, \eqref{normineq} holds.
\end{proof}

With Lemma \ref{lemdifricex} and Lemma \ref{lemnormineq}, we can prove the following result.
\begin{lem}\label{lemasy}
	If $(A,B)$ is stabilizable and $(A,\sqrt{{Q}})$ is detectable, then, for any  $\Pi\in \mathcal{P}^n$,
	there exists a scalar $\delta=\delta(P^a)>0$ such that  the solution of the differential Riccati equation \eqref{dreq}
	satisfies $\lim\limits_{t\to \infty}P(t)=P^*$ for any $P(0)\in \{P(0)\in \mathcal{S}^n|~ ||P(0)-\Pi|| < \delta \}$.
\end{lem}
\begin{proof}
	By Lemma \ref{lemdifricex}, $\lim\limits_{t\to \infty}P(t)=P^*$ if $P(0)$ is such that
	\begin{align*}
		x^T[{P}^aP(0){P}^a+{P}^a]x>0 ~~~ \textup{for all } x \textup{ such that } P^ax\neq 0.
	\end{align*}
	Firstly, we have
	\begin{align*}
		\begin{split}
			&x^T[{P}^aP(0){P}^a+{P}^a]x\\=&x^T[{P}^a(P(0)-\Pi){P}^a+{P}^a\Pi{P}^a+{P}^a]x \\
			= & x^T{P}^a\Pi{P}^ax+\frac{1}{2}x^T{P}^ax+x^T[\frac{1}{2}P^a+{P}^a(P(0)-\Pi){P}^a]x.
		\end{split}
	\end{align*}
	Thus, to make $x^T[{P}^aP(0){P}^a+{P}^a]x>0$, we just need to make sure $\frac{1}{2}P^a+{P}^a(P(0)-\Pi){P}^a\geq0$.
	
	Let $P^a=T\begin{bmatrix}
		\Lambda & 0\\0 & 0
	\end{bmatrix}T^{-1}$, where $\Lambda>0$ and $T$ is some orthogonal matrix. Then, to make $\frac{1}{2}P^a+{P}^a(P(0)-\Pi){P}^a\geq0$, we just need to ensure $\frac{1}{2}\begin{bmatrix}
		\Lambda & 0\\0 & 0
	\end{bmatrix}+\begin{bmatrix}
		\Lambda & 0\\0 & 0
	\end{bmatrix}T^{-1}(P(0)-\Pi)T\begin{bmatrix}
		\Lambda & 0\\0 & 0
	\end{bmatrix}\geq0$. Let $T^{-1}(P(0)-\Pi)T=\begin{bmatrix}
		U_{11}&U_{12}\\U_{21}& U_{22}
	\end{bmatrix}$, where $U_{11},U_{12},U_{21}$ and $ U_{22}$ are block matrices with compatible dimensions. Then, $\begin{bmatrix}
		\Lambda & 0\\0 & 0
	\end{bmatrix}T^{-1}(P(0)-\Pi)T\begin{bmatrix}
		\Lambda & 0\\0 & 0
	\end{bmatrix}=\begin{bmatrix}
		\Lambda U_{11} \Lambda &0\\0&0
	\end{bmatrix}$. Thus, if $\frac{1}{2}\Lambda+\Lambda U_{11} \Lambda\geq 0$, then $\frac{1}{2}P^a+{P}^a(P(0)-\Pi){P}^a\geq0$, and thereby $x^T[{P}^aP(0){P}^a+{P}^a]x>0 \textup{ for all } x \textup{ such that } P^ax\neq 0$.
	
	To guarantee $\frac{1}{2}\Lambda+\Lambda U_{11} \Lambda\geq 0$, we just need to ensure $\min_{||v||=1}[v^T(\frac{1}{2}\Lambda+\Lambda U_{11}\Lambda)v]\geq 0$. For any $||v||=1$, notice
	\begin{align*}
		&v^T(\frac{1}{2}\Lambda+\Lambda U_{11}\Lambda)v\\
		\geq & \frac{1}{2}v^T\Lambda v-|v^T\Lambda U_{11}\Lambda v|\\
		\geq &  \frac{1}{2} \lambda_{min}(\Lambda)-||\Lambda||^2||U_{11}||\\
		\geq & \frac{1}{2} \lambda_{min}(\Lambda)-||\Lambda||^2||P(0)-\Pi||,
	\end{align*}
	where $ \lambda_{min}(\Lambda)>0$ is the minimum eigenvalue of $\Lambda$ and the last inequality is obtained by using \eqref{normineq} and the relation $||T^{-1}(P(0)-\Pi)T||= ||P(0)-\Pi||$.
	
	Choose $\delta=\frac{\lambda_{min}(\Lambda)}{2||\Lambda||^2}>0$. Then, {\color{black}for any $P(0) \in \mathcal{S}^n$ such that $||P(0)-\Pi|| < \delta$,}
	{\color{black}\begin{align*}
		&v^T(\frac{1}{2}\Lambda+\Lambda U_{11}\Lambda)v\\
		\geq & \frac{1}{2} \lambda_{min}(\Lambda)-||\Lambda||^2||P(0)-\Pi||\\
		> & \frac{1}{2} \lambda_{min}(\Lambda)-||\Lambda||^2\delta\\
		= & 0.
	\end{align*}}The proof is thus completed.

    {\color{black}Since $P^a=T\begin{bmatrix}
	\Lambda & 0\\0 & 0
\end{bmatrix}T^{-1}$,  $\Lambda$ is a function of $P_a$. Thus,  $\delta=\frac{\lambda_{min}(\Lambda)}{2||\Lambda||^2}$ is also determined by $P_a$. That is, $\delta=\delta(P^a)$ is also a function of $P_a$.}
\end{proof}

\begin{rem}\label{remasy}
	 Take $\Pi=P^*$, then Lemma \ref{lemasy} implies that $P^*\geq 0$ is{ locally asymptotically convergent}. Moreover,  Lemma \ref{lemasy} indicates $\mathcal{P}^n$ is in the interior of the domain of attraction of $P^*$. Thus, any $P(0)\geq 0$ is in the interior of the domain of attraction of $P^*$.
\end{rem}

The following lemma is extended from Lemma 3.4 of \cite{bian2016}.
\begin{lem}\label{lembounded}
	Assume $(A,B)$ is stabilizable and $(A,\sqrt{{Q}})$ is detectable. Then, for any ${P}^V_0\geq0$,  given $\{{P}^V_k\}_{k=0}^{\infty}$ defined by Algorithm \ref{vialg1ex}, there exist $N\in \mathbb{N}_+$ and a compact set $\mathcal{K}\subset \mathcal{P}^{n}$ with nonempty interior, such that $\{{P}^V_k\}_{k=N}^{\infty}\subset \mathcal{K}$ and $P^*\in \mathcal{K}\subset R_A$, where $R_A$ is the region of attraction of $P^*$.
\end{lem}
\begin{proof}
	The proof of this lemma is almost the same as that of Lemma 3.4 in \cite{bian2016}. Even though \cite{bian2016} only discussed the case where $(A,B)$ is stabilizable and $(A,\sqrt{{Q}})$ is observable, their proof still works for the cases with $(A,B)$ stabilizable and $(A,\sqrt{{Q}})$ detectable. This is because Lemma 3.4 of \cite{bian2016} is bulit on the locally {\color{black}asymptotical} stability of $P^*$ and the fact that $P(0)>0$ is in the interior of the region of attraction of $P^*$ when $(A,B)$ is stabilizable and $(A,\sqrt{{Q}})$ is observable, which has been guaranteed by Lemma \ref{lemdifric}. By Lemma \ref{lemasy},  $P^*$ is also locally asymptotically stable and any $P(0)\geq 0$ is in the interior of the region of attraction of $P^*$ when $(A,B)$ is stabilizable and $(A,\sqrt{{Q}})$ is detectable. The rest of the proof is just the same as that of Lemma 3.4 in \cite{bian2016}.
\end{proof}

With the preparation above, we are now ready to establish the convergence of Algorithm \ref{vialg1ex} as follows.
\begin{thm}\label{thmconvi}
	If $(A,B)$ is stabilizable and $(A,\sqrt{{Q}})$ is detectable, then, for any ${P}^V_0\geq0$, the sequence $\{{P}^V_k\}_{k=0}^{\infty}$ obtained from Algorithm \ref{vialg1ex} satisfies $\lim\limits_{k\to \infty}{P}^V_k={P}^*$, where $P^*$ is the unique positive semi-definite solution to \eqref{areli}.
\end{thm}

\begin{proof}
	The proof of Theorem \ref{thmconvi} follows the same lines as those of Theorem 3.3 in \cite{bian2016}.
	
	First, let $Z_k\triangleq P_{k+1}^V-\tilde{P}_{k+1}^V$. Then, based on Lemma \ref{lembounded}, there exist a sufficiently large integer $N$ and a compact set $\mathcal{K}$ such that $P_{k}^V\in \mathcal{K}$ and
	$$
	Z_k=\left\{
	\begin{aligned}
		&P_0^V-\tilde{P}_{k+1}^V, \,\, & \textup{if } \tilde{P}_{k+1}^V\notin \mathcal{K} \\
		&0\,\,& \textup{Otherwise }
	\end{aligned}
	\right.
	$$ for any $k\geq N$.
	
	Consider the following continuous-time interpolation:
	\begin{align*}
		P^0(t)=
		&P_k^V, \,\,  t\in [t_k,t_{k+1}), k\in \mathbb{N},
	\end{align*}
	where $t_0=0$ and $t_k=\sum_{i=0}^{k-1}\epsilon_i$ for $k\geq 1$. Define the shifted process $P^k(t)=P^0(t_k+t), \forall t\geq 0,k\in \mathbb{N}$. Then, for all $k\geq N$ and $t\geq 0$,
	\begin{align*}
		P^k(t)=&P_k^V+\sum_{i=k}^{m(t+t_k)-1} \epsilon_i (A^T{P}^V_i+{P}^V_iA-{P}^V_iB{R}^{-1}B^T{P}^V_i\\&+{Q})+\sum_{i=k}^{m(t+t_k)-1}Z_i\\
		=& P_k^V+H^k(t)+Z^k(t)+e^k(t),
	\end{align*}
	where
	\begin{align*}
		H^k(t)=&\int_{0}^{t}(A^T{P}^k(s)+{P}^k(s)A-{P}^k(s)B{R}^{-1}B^T{P}^k(s)\\&+{Q})ds\\
		Z^k(t)=&\sum_{i=k}^{m(t+t_k)-1}Z_i\\
		e^k(t)=&\sum_{i=k}^{m(t+t_k)-1} \epsilon_i (A^T{P}^V_i+{P}^V_iA-{P}^V_iB{R}^{-1}B^T{P}^V_i\\&+{Q})-H^k(t)\\
		m(t)=&j, ~~0\leq t_j\leq t<t_{j+1}.
	\end{align*}
	As in the proof of Theorem 3.3 in \cite{bian2016}, one can show that $\{H^k(t)\}_{k=N}^{\infty}$, $\{Z^k(t)\}_{k=N}^{\infty}$ and $\{e^k(t)\}_{k=N}^{\infty}$ are all relatively compact in $\mathcal{D}([0,T], \mathcal{S}^n)$ for any $T>0$, and $\lim\limits_{k\to \infty}e^k(t)=0$. Then, $\{P^k(t)\}_{k=N}^{\infty}$ is also relatively compact in $\mathcal{D}([0,T], \mathcal{S}^n)$ for any $T>0$. Thus, there exists a converging subsequence $\{(P^{k'}(t),Z^{k'}(t))\}$ of $\{(P^{k}(t),Z^{k}(t))\}$ whose limit $(P(t),Z(t))$ satisfies, for any $0 \leq t \leq T$,
	\begin{align}\label{pint}
		P(t)=&P(0)+\int_{0}^{t}(A^T{P}(s)+{P}(s)A+{Q}\notag\\&-{P}(s)B{R}^{-1}B^T{P}(s))ds+Z(t).
	\end{align}
	
	Moreover, as in the proof of Lemma 3.4 in \cite{abounadi2002}, it can be shown that $Z(t)=0$. Thus,
	taking the derivative on both sides of \eqref{pint} directly leads to \eqref{dreq}.
	Since $P^*\in \mathcal{K}\subset R_A$ is locally asymptotically stable and $\{{P}^V_k\}_{k=N}^{\infty}$  remains in the region of attraction of $P^*$, we conclude $\lim\limits_{k\to \infty}{P}^V_k={P}^*$ by Theorem 2.1 of Chapter 5 in \cite{kushner2003}.
\end{proof}

\vskip 0pt plus -1fil

\begin{IEEEbiography}[{\includegraphics[width=1in,height=1.25in,clip,keepaspectratio]{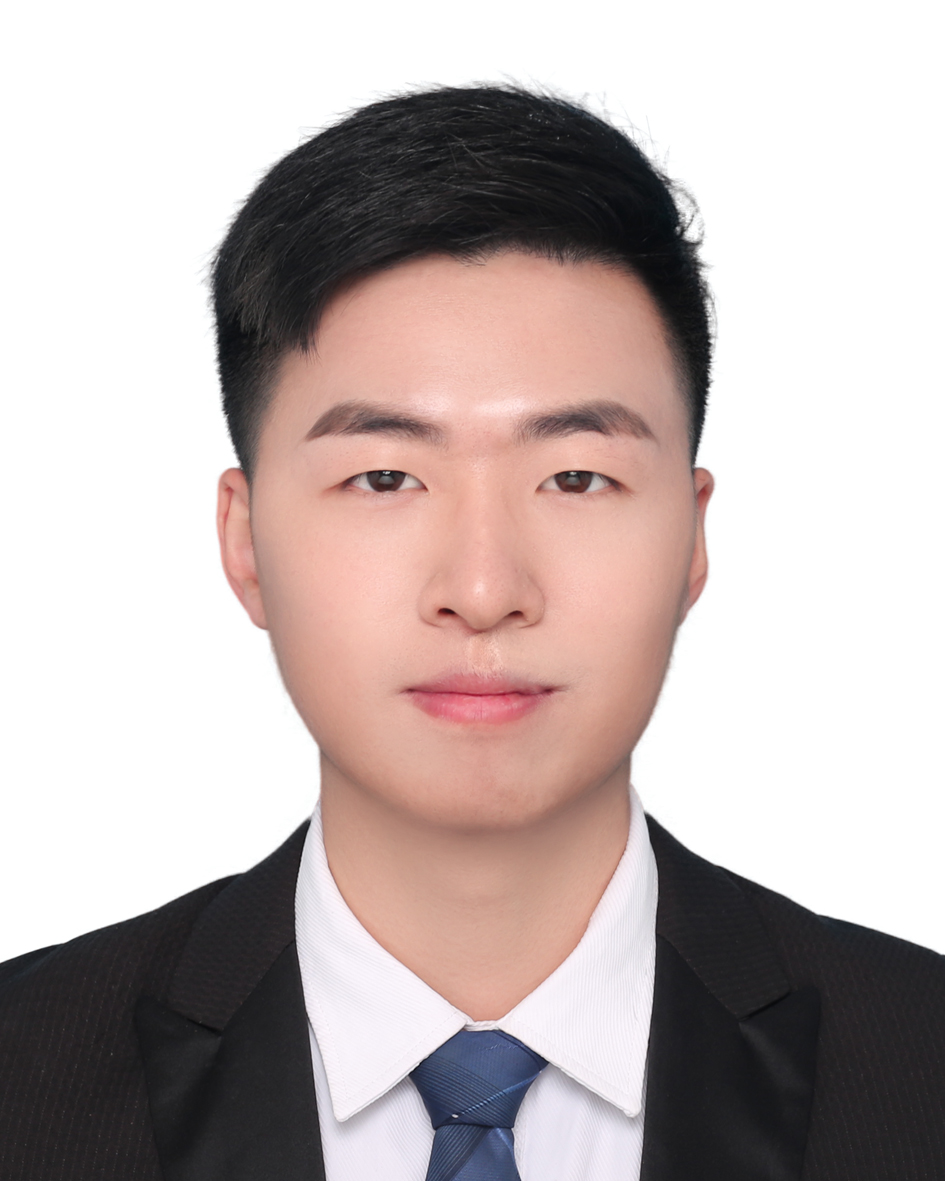}}]
	{Liquan Lin} received the B.Eng. degree from Huazhong University of Science and Technology, Wuhan, China,  in 2021. He is currently pursuing the Ph.D. degree in the Department of Mechanical and Automation Engineering, The Chinese University of Hong Kong, Hong Kong SAR, China.
	
	His current research interests include multi-agent systems, output regulation, data-driven control, and reinforcement learning.
	
\end{IEEEbiography}

\begin{IEEEbiography}[{\includegraphics[width=1in,height=1.25in,clip,keepaspectratio]{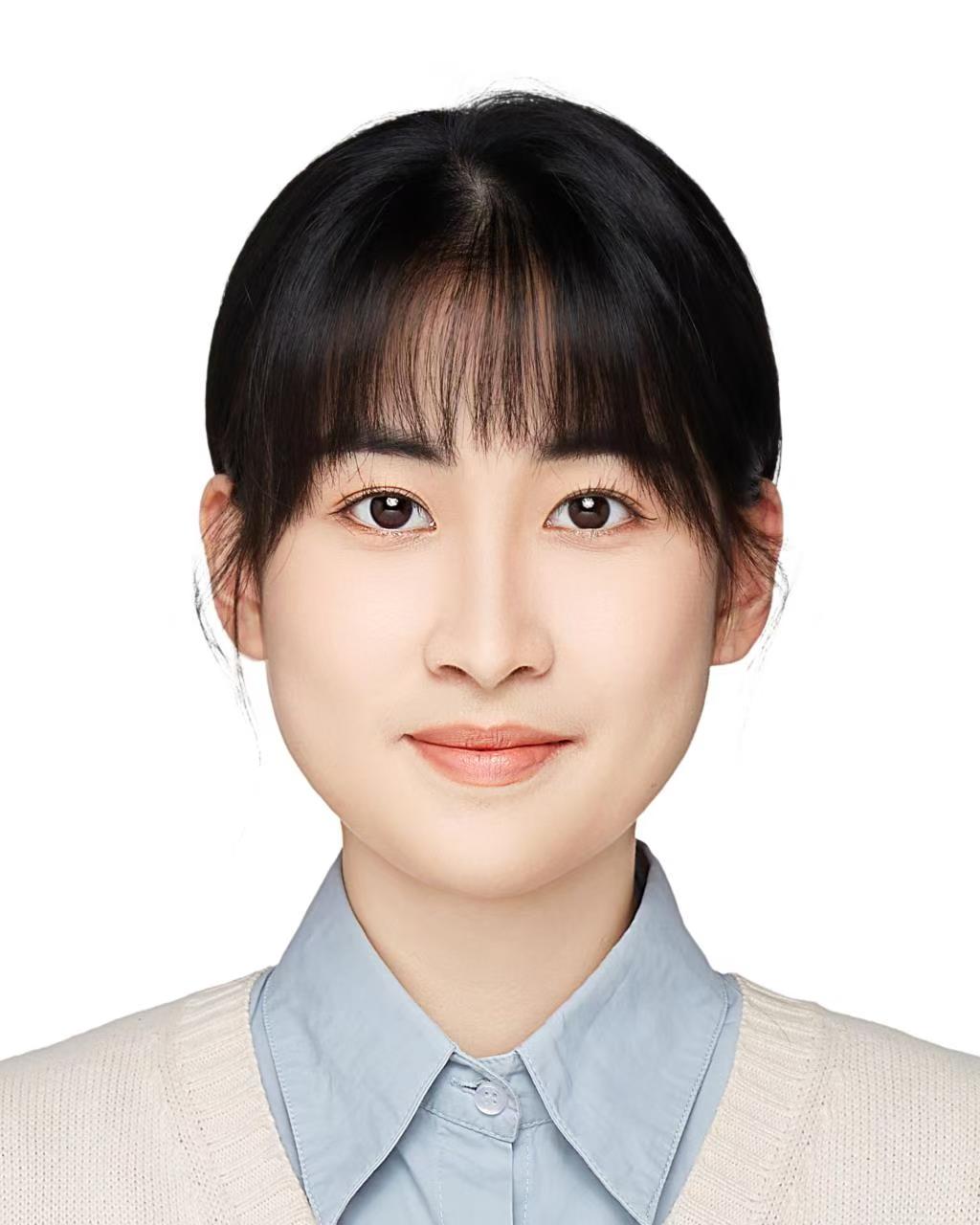}}]
	{Haoyan Lin} received her B.Eng.\ degree in Automation Science and Engineering from South China University of Technology, Guangzhou, China, in 2022. She is currently pursuing the Ph.D.\ degree in the Department of Mechanical and Automation Engineering, The Chinese University of Hong Kong, Hong Kong SAR, China. Her current research interests include Euler-Lagrange systems, output regulation and data-driven control.
	
\end{IEEEbiography}

\begin{IEEEbiography}[{\includegraphics[width=1in,height=1.25in,clip,keepaspectratio]{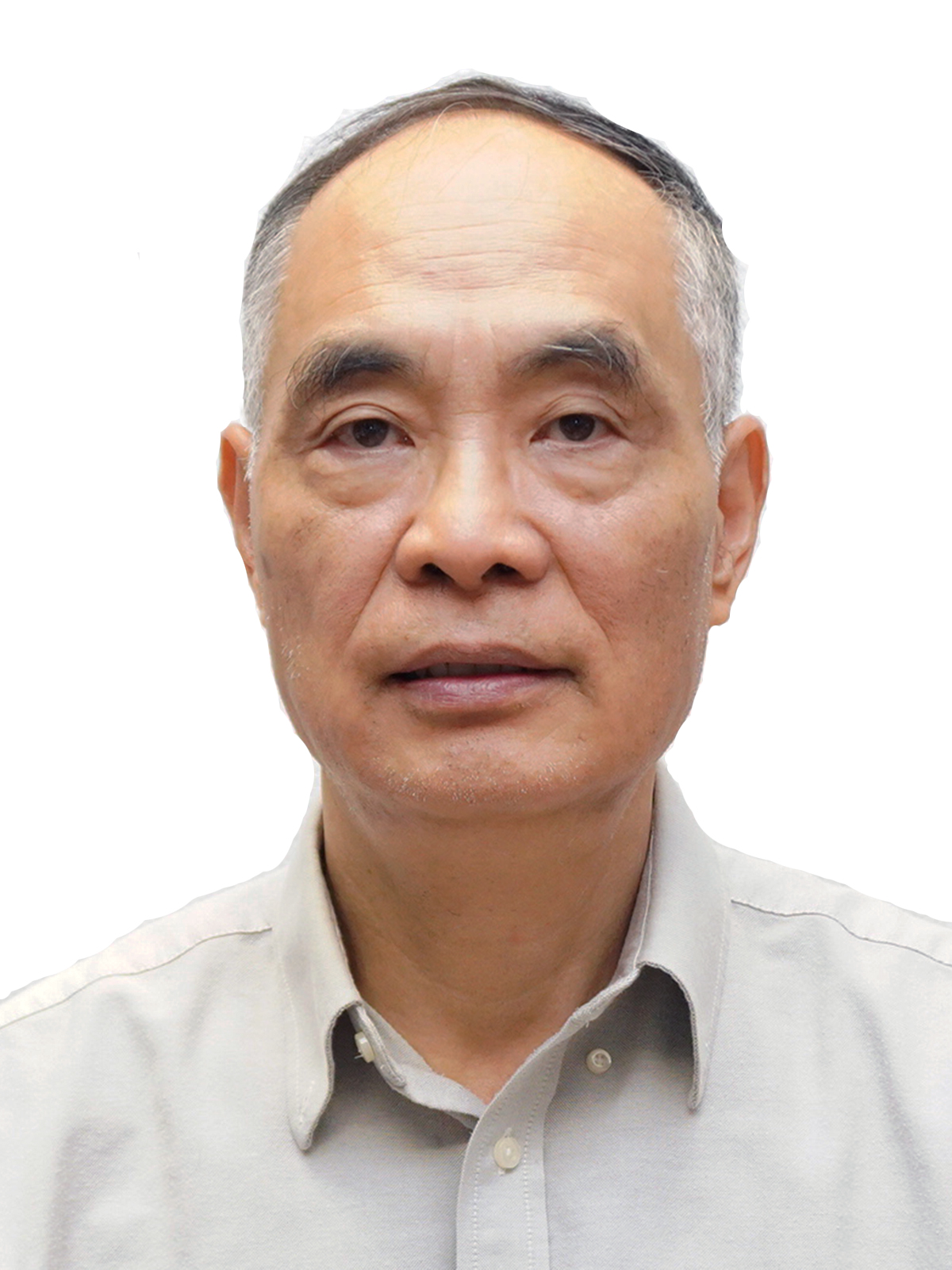}}]{Jie Huang} (Life Fellow, IEEE) received the Diploma
	from Fuzhou University, Fuzhou, China,
	the master’s degree from Nanjing University
	of Science and Technology, Nanjing, China, and
	the Ph.D. degree from Johns Hopkins University,
	Baltimore, MD, USA.
	
	He is a Choh-Ming Li Research Professor of
	of Mechanical and Automation Engineering, The
	Chinese University of Hong Kong (CUHK), and Associate Dean (Research) of Faculty of Engineering, CUHK.
	His research interests include nonlinear control theory and applications, multi-agent systems, game theory, and flight guidance and control.
	
	Dr. Huang is a Fellow of IFAC, CAA, and HKIE.
\end{IEEEbiography}


\begin{thebibliography}{99}
	
	\bibitem{sutton2018reinforcement}
	R. S. Sutton and A. G. Barto,
	\emph{Reinforcement learning: An introduction},
	MIT press, 2018.
	
	
	\bibitem{bertsekas2019reinforcement}
	D. Bertsekas,
	\emph{Reinforcement learning and optimal control},
	Athena Scientific, 2019.
	
	\bibitem{lewis2012}
	F. L. Lewis, D. Vrabie and V. L. Syrmos,
	\emph{Optimal control},
	John Wiley \& Sons, 2012.
	
	\bibitem{vrabie2009adaptive}
	D. Vrabie, O. Pastravanu, M. Abu-Khalaf, and F. L. Lewis,
	``Adaptive optimal control for continuous-time linear systems based on policy iteration,''
	\emph{Automatica}, vol. 45, no. 2, pp. 477-484, 2009.
	
	\bibitem{jiang2012computational}
	Y. Jiang and Z. P. Jiang,
	``Computational adaptive optimal control for continuous-time linear systems with completely unknown dynamics,''
	\emph{Automatica}, vol. 48, no. 10, pp. 2699-2704, 2012.
	
	{\color{black}\bibitem{lopez2023}
	V. G. Lopez and M. A. Müller, ``An efficient off-policy reinforcement learning algorithm for the continuous-time LQR problem," in \emph{Proc. IEEE Conf. Decis. Control},  2023,
	pp. 13-19.}
	
	\bibitem{bian2016}
	T. Bian and Z. P. Jiang,
	``Value iteration and adaptive dynamic pro- gramming for data-driven adaptive  optimal control design,''
	\emph{Automatica}, vol. 71, pp. 348–360, 2016.
	
	\bibitem{gao2016adaptive}
	W. Gao and Z. P. Jiang,
	``Adaptive dynamic programming and adaptive optimal output regulation of linear systems,''
	\emph{IEEE Trans. Autom. Control}, vol. 61, no. 12, pp. 4164-4169, 2016.
	
	\bibitem{jiang2022value}
	Y. Jiang, W. Gao, J. Na, D. Zhang, T. T. Hämäläinen, V. Stojanovic and F. L. Lewis,
	``Value iteration and adaptive optimal output regulation with assured convergence rate,''
	\emph{Control Eng. Pract.}, vol. 121, pp. 105042, 2022.
	
	\bibitem{lin2023}
	L. Lin and J. Huang,
	``Refined Algorithms for Adaptive Optimal Output Regulation and Adaptive Optimal Cooperative Output Regulation Problems,''
	\emph{IEEE Trans. Control Netw. Syst.},
	 vol. 12, no. 1, pp. 241-250, 2024.
	
	
	
	\bibitem{chen2022homotopic}
	C. Chen, F. L. Lewis and B. Li,
	``Homotopic policy iteration-based learning design for unknown linear continuous-time systems,''
	\emph{Automatica}, vol. 138, pp. 110153, 2022.
	
	\bibitem{xie2022optimal}
	K. Xie, X. Yu and W. Lan,
	``Optimal output regulation for unknown continuous-time linear systems by internal model and adaptive dynamic programming,''
	\emph{Automatica}, vol. 146, pp. 110564, 2022.
	
	\bibitem{lin2024auto}
	L. Lin and J. Huang,
	``Distributed Adaptive Cooperative Optimal Output Regulation via Integral Reinforcement Learning,''
	\emph{Automatica}, vol. 170, pp. 111861, 2024.
	
	\bibitem{chen2020off}
	C. Chen, F. L. Lewis, K. Xie, S. Xie and Y. Liu,
	``Off-policy learning for adaptive optimal output synchronization of heterogeneous multi-agent systems,''
	\emph{Automatica}, vol. 119, pp. 109081, 2020.
	
	\bibitem{zhu2014}
	L. M. Zhu, H. Modares, G. O. Peen, F. L. Lewis and B. Yue,
	``Adaptive suboptimal output-feedback control for linear systems using integral reinforcement learning,''
	\emph{IEEE Trans. Control Syst. Technol.}, vol. 23, no. 1, pp. 264-273, 2014.
	
	\bibitem{modares2016}
	H. Modares, F. L. Lewis and Z. P. Jiang,
	``Optimal output-feedback control of unknown continuous-time linear systems using off-policy reinforcement learning,''
	\emph{IEEE Trans. Cybern.}, vol. 46, no. 11, pp. 2401-2410, 2016.
	
	\bibitem{rizvi2019}
	S. A. A. Rizvi and Z. Lin,
	``Reinforcement learning-based linear quadratic regulation of continuous-time systems using dynamic output feedback,''
	\emph{IEEE Trans. Cybern.}, vol. 50, no. 11, pp. 4670-4679, 2019.
	
	\bibitem{rizvi2023}
	S. A. A. Rizvi and Z. Lin,
	\emph{Output Feedback Reinforcement Learning Control for Linear Systems}, Birkhäuser, 2023.
	
	\bibitem{cui2023}
	L. Cui and Z. P. Jiang,
	``Learning-Based Control of Continuous-Time Systems Using Output Feedback,''
	in \emph{Proc. SIAM Conf. Control Appl.}, 2023, pp. 17-24.
	
	\bibitem{chen2023output}
	C. Chen, F. L. Lewis, K. Xie, Y. Lyu and S. Xie,
	``Distributed output data-driven optimal robust synchronization of heterogeneous multi-agent systems,''
	\emph{Automatica}, vol. 153, pp. 111030, 2023.
	
	\bibitem{xie2024auto}
	K. Xie, Y. Zheng, Y. Jiang, W. Lan and X. Yu,
	``Optimal dynamic output feedback control of unknown linear continuous-time systems by adaptive dynamic programming,''
	\emph{Automatica}, vol. 163, pp. 111601, 2024.
	
	\bibitem{xie2023output}
	K. Xie, Y. Zheng, W. Lan and X. Yu,
	``Adaptive optimal output regulation of unknown linear continuous-time systems by dynamic output feedback and value iteration,''
	\emph{Control Eng. Pract.}, vol. 141, pp. 105675, 2023.
	
		\bibitem{Kleinman} D. Kleinman, ``On an iterative technique for riccati equation computations,''
	\emph{IEEE Trans. Autom. Control}, vol. 13, no. 1, pp. 114115, 1968.
	
	\bibitem{kucera1972}
	V. Kucera,
	``A contribution to matrix quadratic equations,''
	\emph{IEEE Trans. Autom. Control}, vol. 17, no. 3, pp. 344-347, 1972.
	
	{\color{black}\bibitem{luenberger}
		D. G. Luenberger, \emph{Dynamic Systems}, 1979.}
	

	
	\bibitem{ctchenbook}
	C. T. Chen,
	\emph{Linear system theory and design}, Oxford University Press, 1999.
	
	\bibitem{cai2017}
	H. Cai, F. L. Lewis, G. Hu and J. Huang,
	``The adaptive distributed observer approach to the cooperative output regulation of linear multi-agent systems,''
	\emph{Automatica}, vol. 75, pp. 299-305, 2017.
	
	\bibitem{kleinmanphd}
	D. Kleinman,
	``Suboptimal design of linear regulator systems subject to computer storage limitations,''
	PhD thesis, Massachusetts Institute of Technology, 1967.
	
	\bibitem{Glad2010book}
	T. Glad and L. Ljung,
	\emph{Control theory}, CRC press, 2010.
	
	\bibitem{kucera1973}
	V. Kucera,
	``A review of the matrix Riccati equation,''
	\emph{Kybernetika}, vol. 9, no. 1, pp. 42-61, 1973.
	
	\bibitem{kailath1976}
	T. Kailath and L. Ljung,
	``The asymptotic behavior of constant-coefficient Riccati differential equations,''
	\emph{IEEE Trans. Autom. Control}, vol. 21, no. 3, pp. 385-388, 1976.
	
	\bibitem{bian2019}
	T. Bian and Z. P. Jiang,
	``Continuous-time robust dynamic programming,''
	\emph{SIAM J. Control. Optim.}, vol. 57, no. 6, pp. 4150-4174, 2019.
	
	\bibitem{khalilbook}
	H. K. Khalil,
	\emph{Nonlinear systems (3rd Ed.)}, Prentice Hall, 2002.
	
	\bibitem{abounadi2002}
	J. Abounadi, D. P. Bertsekas and V. Borkar,
	``Stochastic approximation for nonexpansive maps: Application to Q-learning algorithms,''
	\emph{SIAM J. Control. Optim.},
	vol. 41, no. 1, pp.1-22, 2002.
	
	\bibitem{kushner2003}
	H. J. Kushner and G. G. Yin,
	\emph{Stochastic Approximation and Recursive Algorithms and Applications},
	Springer, 2003.
	
\end{thebibliography}
\end{document}